\documentclass[11pt]{article}
\usepackage{amssymb, amsmath, fullpage, amsthm, pictexwd,mathrsfs}
\usepackage[PostScript=dvips,balance,midshaft,nohug]{diagrams}
\usepackage{epic,eepic}
\usepackage{float}

\newtheorem{theorem}{Theorem}[section]

\newtheorem{corollary}[theorem]{Corollary}
\newtheorem{definition}[theorem]{Definition}
\newtheorem{lemma}[theorem]{Lemma}
\newtheorem{proposition}[theorem]{Proposition}
\newtheorem{example}[theorem]{Example}
\newtheorem{remark}[theorem]{Remark}

\numberwithin{equation}{section}

\newcommand{\Bbbb}{{\mathbb B}}
\newcommand{\Ccc}{{\mathbb C}}
\newcommand{\Nnn}{{\mathbb N}}
\newcommand{\Rrr}{{\mathbb R}}
\newcommand{\Sss}{{\mathbb S}}
\newcommand{\Ttt}{{\mathbb T}}
\newcommand{\Zzz}{{\mathbb Z}}

\DeclareMathOperator{\Semi}{Semi}
\DeclareMathOperator{\Pyr}{Pyr}
\DeclareMathOperator{\link}{link}
\DeclareMathOperator{\cone}{cone}
\DeclareMathOperator{\wt}{wt}

\newcommand{\zetabar}{\bar{\zeta}}
\newcommand{\mubar}{\bar{\mu}}
\newcommand{\fbar}{\bar{f}}
\newcommand{\hhbar}{\bar{h}}
\newcommand{\shellcomponent}{\check{\Phi}}

\newcommand{\zab}{\Zzz\langle\av,\bv\rangle}

\newcommand{\onethingatopanother}[2]
{\genfrac{}{}{0pt}{}{#1}{#2}}

\newcommand{\tensor}{\otimes}
\newcommand{\coveredby}{\prec}

\newcommand{\ab}{\av\bv}
\newcommand{\av}{{\bf a}}
\newcommand{\bv}{{\bf b}}
\newcommand{\cv}{{\bf c}}
\newcommand{\dv}{{\bf d}}
\newcommand{\cd}{\cv\dv}
\newcommand{\ctd}{{\cv\text{-}2\dv}}

\newcommand{\hz}{\widehat{0}}
\newcommand{\ho}{\widehat{1}}

\newcommand{\cups}{\cup \cdots \cup}

\newcommand{\meet}{\wedge}
\newcommand{\longto}{\longrightarrow}

\newcommand{\mc}{\cv}

\newcommand{\mcc}{\mc^{2}}
\newcommand{\md}{\dv}

\newcommand{\mccc}{\mc^{3}}
\newcommand{\mcd}{\mc\md}
\newcommand{\mdc}{\md\mc}

\newarrow{Equals} =====
\newarrow{Implies} ===={=>}
\newarrow{Onto} ----{>>}
\newarrow{Into} C--->
\newarrow{IntoA} {hooka}--->
\newarrow{IntoB} {hookb}--->
\newarrow{Dotsto} ....>
\newarrowhead{sp}{\ }{\ }{\ }{\ }
\newarrowtail{sp}{\ }{\ }{\ }{\ }
\newarrow{Line}----{-}
\newarrow{Eq}{sp}==={sp}
\newarrowtail{eev}\Leftarrow\Rightarrow\Uparrow\Downarrow
\newarrow{Equiv} {eev}==={=>}

\newcommand{\quash}[1]{}

\newcommand{\bR}{\mathbf R}
\newcommand{\bT}{\mathbf T}

\begin{document}

\title{Euler flag enumeration of Whitney stratified spaces}

\author{{\sc Richard EHRENBORG,}
        {\sc Mark GORESKY}
        and
        {\sc Margaret READDY}}

\date{}

\maketitle

\begin{abstract}
The flag vector contains all the face incidence data of a polytope,
and in the poset setting, the chain enumerative data.  It is a
classical result due to Bayer and Klapper that for face lattices of
polytopes, and more generally, Eulerian graded posets, the flag vector
can be written as a $\cd$-index, a non-commutative polynomial which
removes all the linear redundancies among the flag vector entries.
This result holds for regular $CW$~complexes.

We relax the regularity condition to show the $\cd$-index exists for
Whitney stratified manifolds by extending the notion of a graded poset
to that of a quasi-graded poset.  This is a poset endowed with an
order-preserving rank function and a weighted zeta function.  This
allows us to generalize the classical notion of Eulerianness, and
obtain a $\cd$-index in the quasi-graded poset arena.  We also extend the
semi-suspension operation to that of embedding a complex in the
boundary of a higher dimensional ball and study the simplicial shelling
components.

\vspace*{2 mm}


\vspace*{2 mm}

\end{abstract}

\section{Introduction}

In this paper we extend the theory of face incidence
enumeration of polytopes, and more generally,
chain enumeration in graded
Eulerian posets, to that of Whitney stratified spaces
and quasi-graded posets.

The idea of enumeration using
the Euler characteristic was 
suggested
throughout Rota's work
and
influenced by
Schanuel's
categorical viewpoint~\cite{Klain_Rota,Rota, Rota_Foundations_I, Schanuel}.
In order to carry out such
a program that is topologically meaningful
and which captures the broadest possible classes
of examples,
two key insights are required.  First,
the notion of grading in the face lattice of a polytope
must be relaxed.
Secondly, 
the usual
zeta function in the incidence algebra must be extended
to include the Euler
characteristic as an important instance.

Recall an {\em Eulerian poset}
is
a graded partially ordered set (poset) $P$ such that 
every nontrivial interval
satisfies the Euler--Poincar\'e relation,
that is,
the number of elements of even rank equals the
number of elements of odd rank.
Equivalently,
the M\"obius function
is given by 
$\mu(x,y) = (-1)^{\rho(y)-\rho(x)}$,
where $\rho$ denotes the {\em rank function}.
Another way to state this {\em Eulerian condition}
is that the inverse of the zeta function $\zeta(x,y)$
(where $\zeta(x,y) = 1$ if $x \leq y$
and $0$ otherwise),
that is, the M\"obius function $\mu(x,y)$,
is
given by the function $(-1)^{\rho(y)-\rho(x)} \cdot \zeta(x,y)$.
Families of Eulerian
posets include 
(a) the face lattice
of a convex polytope,
(b) the face poset of a regular cell decomposition of a
homology sphere, and (c) the elements of a finite Coxeter group 
ordered by the strong
Bruhat order. In the case of a
convex polytope the Eulerian condition expresses the fact that the
link of each face has the Euler characteristic of a sphere.

The {\em $f$-vector} of a convex polytope
enumerates, for
each non-negative integer $i$,
the number $f_{i}$ of $i$-dimensional faces in the polytope.  
It satisfies the
Euler--Poincar\'e relation.  
The problem of understanding the $f$-vectors of polytopes
harks back to Steinitz~\cite{Steinitz},
who completely described the
$3$-dimensional case.
For polytopes
of dimension greater than three 
the problem is still open.
For {\em simplicial} polytopes,
that is, each $i$-dimensional face is an
$i$-dimensional simplex,
the $f$-vectors satisfy linear relations
known as the Dehn--Sommerville relations.
Furthermore,
the $f$-vectors of simplicial polytopes have been
completely characterized by work of McMullen~\cite{McMullen},
Billera and Lee~\cite{Billera_Lee}, and Stanley~\cite{Stanley_g}.

The {\em flag $f$-vector}
of a graded poset counts
the number of chains 
passing through a prescribed set of ranks.
In the case of a polytope,
it records all of the face incidence data,
including that of the $f$-vector.
Bayer and Billera
proved that the flag $f$-vector of any Eulerian poset satisfies a
collection of linear equalities now known as the
{\em generalized Dehn--Sommerville relations}~\cite{Bayer_Billera}.
These linear equations may be
interpreted as natural redundancies among
the components of the flag $f$-vector.
Bayer and
Klapper removed these redundancies by showing that the space of flag 
$f$-vectors of Eulerian posets has a
natural basis consisting of
non-commutative polynomials in the two variables $\cv$
and $\dv$~\cite{Bayer_Klapper}.
The coefficients of this {\em ${\cd}$-index} were later shown by Stanley
to be non-negative in the case of spherically-shellable posets~\cite{Stanley_d}.
Other milestones
for the ${\cd}$-index include its inherent coalgebraic
structure~\cite{Ehrenborg_Readdy},
its appearance in the proofs of inequalities for flag
vectors~\cite{Billera_Ehrenborg,Ehrenborg_lifting,Ehrenborg_Karu,Karu},
its use in understanding the combinatorics of arrangements of subspaces and
sub-tori~\cite{Billera_Ehrenborg_Readdy_om,Ehrenborg_Readdy_Slone}, 
and most recently, its connection to the Bruhat graph
and Kazhdan--Lusztig theory~\cite{Billera_Brenti,Ehrenborg_Readdy_Bruhat}.

In this article we extend the ${\cd}$-index and its properties
to a more general situation,
that of quasi-graded posets and Whitney stratified spaces.  A quasi-grading 
on a poset $P$ consists of a strictly 
order-preserving ``rank'' function 
$\rho : P \longrightarrow \Nnn$ and
a weighted zeta function $\zetabar$ in the incidence 
algebra $I(P)$ such that $\zetabar(x,x) = 1$ for all $x \in P$.
See Section~\ref{section_quasi_graded}.
A quasi-graded poset $(P, \rho, \zetabar)$ will be said to be {\em Eulerian}
if the function $(-1)^{\rho(y)-\rho(x)} \cdot \zetabar(x,y)$
is the inverse of
$\zetabar(x,y)$ in the incidence algebra of~$P$.
This reduces to the classical definition of Eulerian
 if $(P,\rho,\zetabar)$ 
is a ranked 
poset with the  standard
zeta function
$\zeta$.

We show that
Eulerian $\zetabar$ functions exist on most posets
(Proposition~\ref{proposition_odd_rank}).
Let $P$ be a poset 
with a strictly order-preserving ``rank'' function $\rho$.  
Choose $\zetabar(x,y)$ arbitrarily whenever $\rho(y)-\rho(x)$ 
is odd.  Then {\em there is a unique way to assign values to $\zetabar(x,y)$ 
whenever $\rho(y)-\rho(x)$ is even such that $(P,\rho,\zetabar)$ is an
Eulerian quasi-graded poset}.
Theorem~\ref{theorem_cd} states that the $\cd$-index is
defined for Eulerian quasi-graded posets.
This result is equivalent to
the flag $\fbar$-vector of an Eulerian quasi-graded
poset satisfies the generalized Dehn--Sommerville relations
(Theorem~\ref{theorem_generalized_Dehn--Sommerville}).

Additional properties of quasi-graded posets are discussed in
Sections~\ref{section_Euler},
\ref{section_cd}
and~\ref{section_poset_operations}.
The Alexander duality formula has a natural generalization
to Eulerian quasi-graded posets
(Theorem~\ref{theorem_Alexander}
and Corollary~\ref{corollary_h_symmetry}).
The ${\ab}$-index and the $\cd$-index
of the Cartesian product $P \times B_1$ of a 
quasi-graded poset $P$ and the Boolean algebra $B_1$ is given 
(Proposition~\ref{proposition_pyramid}).  
The ${\ab}$-index of the Stanley product of two quasi-graded 
posets is the product of their ${\ab}$-indexes
(Lemma~\ref{lemma_Stanley_product}).
The ``zipping'' operation (see~\cite{Reading}) can be defined for quasi-graded
posets and the resulting $\ab$-index is calculated
(Proposition~\ref{proposition_zipping}).
Furthermore, the Eulerian property is 
preserved under the zipping operation
(Theorem~\ref{theorem_zipping_preserves_Eulerian}).
Merging strata in a Whitney
stratified manifold,
the geometric analogue of poset zipping,
is established later in 
Section~\ref{section_merging_strata}.

Eulerian ranked posets arise geometrically as the face 
posets of regular
cell decompositions of a sphere,
whereas  Eulerian quasi-graded posets arise
geometrically from the more general case of Whitney stratifications.  
A {\em Whitney stratification} $X$ of 
a compact topological space $W$ is a decomposition of $W$ 
into finitely many smooth manifolds which satisfy Whitney's ``no-wiggle'' 
conditions on how the strata fit together.
See Section~\ref{sec-stratified-sets}.  These conditions 
guarantee (a) that $X$ does not exhibit Cantor set-like behavior and 
(b) that the closure of each stratum is a union of strata.  
The faces of a convex polytope and the cells of a regular
cell complex are examples of Whitney stratifications, but in general, 
a stratum in a stratified space need not be contractible.  Moreover, the 
closure of a stratum of dimension $d$ does not necessarily contain strata of 
dimension $d-1$, or for that matter,
of any other dimension.  Natural Whitney stratifications 
exist for real or complex algebraic sets, analytic sets, 
semi-analytic sets and for quotients of smooth manifolds by compact group actions.

The strata of a Whitney stratification (of a topological space $W$) form a 
poset, where the order relation $A<B$ is given by
$A \subset \overline{B}$.
Moreover, this set admits a natural quasi-grading which
is defined by $\rho(A) = \dim(A)+1$ 
and
$\zetabar(A,B) = \chi(\link(A) \cap B)$ whenever $A < B$ are
strata
and $\chi$ is the Euler characteristic.  
See Definition~\ref{definition_main_definition}. 
This is the setting for our
Euler-characteristic enumeration.

Theorem~\ref{theorem_main_theorem} states that {\em the quasi-graded
poset of strata 
of a Whitney stratified set is Eulerian} and therefore its
$\cd$-index is defined and its
flag $\fbar$-vector satisfies the generalized Dehn--Sommerville relations.
The background and results needed for this proof are developed
in 
Sections~\ref{section_properties_of_the_Euler_characteristic},
\ref{section_control_data}
and~\ref{section_proof}.

It is important to
point out that,
unlike the case of polytopes,
the coefficients of the $\cd$-index
of Whitney stratified manifolds
can be negative.
See Examples~\ref{example_one-gon},
\ref{example_manifold},
\ref{example_point_sphere}
and~\ref{example_n_punctured_torus}.
It is our hope that by applying
topological techniques
to stratified manifolds, we will yield
a tractable interpretation of the
coefficients of the $\cd$-index.
This may ultimately 
explain Stanley's  non-negativity results
for spherically shellable posets~\cite{Stanley_d}
and Karu's results for
Gorenstein* posets~\cite{Karu}.

One may also ask what linear inequalities
hold among the entries of the weighted $\fbar$-vector
of a Whitney stratified manifold, equivalently, what linear
inequalities hold among the coefficients of the $\cd$-index?
See the concluding remarks for further details.

In his proof that the $\cd$-index of a polytope
is non-negative, Stanley introduced the notion of semisuspension.
Given a polytopal complex that is homeomorphic to a ball,
the semisuspension adds another facet whose boundary is the
boundary of the ball. The resulting spherical $CW$~complex has
the same dimension,
and the intervals in its face poset are Eulerian~\cite{Stanley_d}.

It is precisely the setting of Whitney stratified manifolds,
and the larger class of Whitney stratified spaces,
which is critical 
in order to study 
face enumeration of the semisuspension in higher dimensional spheres
and more general topologically interesting examples.
In Sections~\ref{section_the_semisuspension} 
and~\ref{section_inclusion_exclusion_semisuspension}
the $n$th semisuspension 
and its $\cd$-index are studied.  
In Theorem~\ref{theorem_intersections},
by using the method of quasi-graded posets,
we are able to give a short proof (that completely avoids
the use of shellings) of a key result of
Billera and Ehrenborg~\cite{Billera_Ehrenborg}
that was needed for their proof that the $n$-dimensional simplex
minimizes the $\cd$-index among all $n$-dimensional polytopes.

In Section~\ref{section_local} 
we establish the Eulerian relation for
the $n$th semisuspension (Theorem~\ref{theorem_local}).
This implies 
one cannot develop a local $\cd$-index
akin to the local $h$-vector
Stanley devised in~\cite{Stanley_local} 
to understand the effect of subdivisions on the $h$-vector.

In Section~\ref{section_shelling_components}
the $\cd$-index of the $n$th
semisuspension of a non-pure shellable simplicial complex is
determined.
The $\cd$-index of the shelling components is shown
to satisfy a recursion involving a derivation
which first appeared in~\cite{Ehrenborg_Readdy}.
By relaxing the notion of shelling,
we furthermore show that the
shelling components satisfy a Pascal type recursion.
This yields new expressions for the shelling components
and illustrates the power of leaving the realm of
regular cell complexes for that of 
Whitney stratified spaces.

We end with open questions and comments
in the concluding remarks.

\section{Quasi-graded posets and their $\ab$-index}
\label{section_quasi_graded}

Recall the {\em incidence algebra} of a poset is the
set of all functions $f : I(P) \rightarrow \Ccc$
where $I(P)$ denotes the set of intervals in the poset.
The multiplication is given by 
$(f \cdot g)(x,y) = \sum_{x \leq z \leq y} f(x,z) \cdot g(z,y)$
and the identity is given by the delta function
$\delta(x,y) = \delta_{x,y}$, where the second delta is
the usual Kronecker delta function
$\delta_{x,y} = 1$ if $x = y$ and zero otherwise.
A poset is said to be {\em ranked} if
every maximal chain
in the poset
has the same length.
This common length is called the {\em rank} of the poset.
A poset is said to be {\em graded} if
it is ranked and has
a minimal element $\hz$
and
a maximal element $\ho$.
For other poset terminology, we refer the
reader to Stanley's treatise~\cite{Stanley_EC_1}.

We introduce the notion of a quasi-graded poset.
This extends the notion of a ranked poset.
\begin{definition}
  A quasi-graded poset $(P,\rho,\zetabar)$
  consists of
  \vspace*{-4 mm}
  \begin{enumerate}
  \item[(i)]
    a finite poset $P$ (not necessarily ranked),

  \item[(ii)]
    a strictly order-preserving function $\rho$ from $P$ to
    $\Nnn$, that is, $x < y$ implies $\rho(x) < \rho(y)$
    and

  \item[(iii)]
    a function $\zetabar$ in the incidence algebra $I(P)$
    of the poset $P$,
    called the {\em weighted zeta function},
    such that $\zetabar(x,x) = 1$
    for all elements $x$ in the poset $P$.
  \end{enumerate}
\end{definition}
\noindent
Observe that we do not require
the poset to have a minimal element or
a maximal element.
Since $\zetabar(x,x) \neq 0$ for
all $x \in P$, 
the function $\zetabar$
is invertible 
in the incidence algebra $I(P)$
and we
denote its  inverse 
by~$\mubar$. 
See Lemma~\ref{lemma_well-known}
for an expression for the inverse.

For $x \leq y$ in a quasi-graded poset $P = (P, \rho, \zetabar)$,
the {\em rank difference function} is given by
$\rho(x,y) = \rho(y) - \rho(x)$.
We say that a quasi-graded poset $(P,\rho,\zetabar)$
with minimal element $\hz$ and maximal element~$\ho$
has rank $n$ if $\rho(\hz,\ho) = n$.
The interval
$[x,y]$ is itself a quasi-graded poset together with
the rank function
$\rho_{[x,y]}(w) = \rho(w) - \rho(x)$
and the weighted zeta function $\zetabar$.

\begin{example}
  {\rm
    The classical example of a quasi-graded
    poset is $(P,\rho,\zeta)$,
    where $P$ is a
    graded poset with rank
    function $\rho$ and we take the weighted
    zeta function to be the usual zeta function $\zeta$ in the
    incidence algebra
    defined by $\zeta(x,y) = 1$ for all intervals
    $[x,y] \subseteq P$. Here the inverse of the zeta function
    is the M\"obius function, denoted by $\mu(x,y)$.
  }
\end{example}

Let $(P,\rho,\zetabar)$ be a quasi-graded
poset with unique minimal element $\hz$ and
unique maximal element~$\ho$.
The assumption of a quasi-graded
poset having a $\hz$ and $\ho$
will be essential in order to define its $\ab$-index
and $\cd$-index.
For a chain
$c = \{x_{0} < x_{1} < \cdots < x_{k}\}$
in the quasi-graded poset~$P$, define~$\zetabar(c)$
to be the product
\begin{equation}
      \zetabar(c)
=
\zetabar(x_{0},x_{1})
\cdot
\zetabar(x_{1},x_{2})
\cdots
\zetabar(x_{k-1},x_{k})   .
\end{equation}
Similarly, for the chain $c$ 
define its {\em weight} to be
$$    \wt(c)
=
(\av-\bv)^{\rho(x_{0},x_{1})-1}
\cdot
\bv
\cdot
(\av-\bv)^{\rho(x_{1},x_{2})-1}
\cdot
\bv
\cdots
\bv
\cdot
(\av-\bv)^{\rho(x_{k-1},x_{k})-1}  ,  $$
where $\av$ and $\bv$ are non-commutative variables
each of degree $1$.
The {\em $\ab$-index of a quasi-graded poset}~$(P,\rho,\zetabar)$ is
\begin{equation}
\Psi(P,\rho,\zetabar)
=
\sum_{c} \zetabar(c) \cdot \wt(c)  ,  
\end{equation}
where the sum is over all chains
starting at the minimal element $\hz$
and ending at the maximal element~$\ho$, that is,
$c = \{\hz = x_{0} < x_{1} < \cdots < x_{k} = \ho\}$.
When the rank function $\rho$
and the weighted zeta function are clear from the context,
we will write the shorter $\Psi(P)$.
Observe that if a quasi-graded poset~$(P,\rho,\zetabar)$
has rank $n+1$ then its $\ab$-index is homogeneous
of degree $n$.

The $\ab$-index depends on the
rank difference function $\rho(x,y)$ but not on the
rank function itself. Hence we may uniformly
shift the rank function
without changing the $\ab$-index. 
Later we will use the 
convention that $\rho(\hz) = 0$.

A different approach to the $\ab$-index is via
the flag $f$- and flag $h$-vectors.
We extend this route by introducing the
flag $\fbar$- and flag $\hhbar$-vectors.
Let $(P,\rho,\zetabar)$ be a quasi-graded poset of rank $n+1$
having a $\hz$ and $\ho$ such that $\rho(\hz) = 0$.
For $S = \{s_{1} < s_{2} < \cdots < s_{k}\}$ a subset
of $\{1, \ldots, n\}$, define
the {\em flag $\fbar$-vector} by
\begin{equation}
         \fbar_{S}
=
\sum_{c} \zetabar(c)  ,
\end{equation}
where the sum is over all chains
$c = \{\hz = x_{0} < x_{1} < \cdots < x_{k+1} = \ho\}$
in $P$
such that $\rho(x_{i}) = s_{i}$ for all $1 \leq i \leq k$.
The  {\em flag $\hhbar$-vector}
is defined by the relation (and by inclusion--exclusion,
we also display its inverse relation)
\begin{equation}
      \hhbar_{S} = \sum_{T \subseteq S} (-1)^{|S-T|} \cdot \fbar_{T}
\:\:\:\: \mbox{ and } \:\:\:\:
\fbar_{S} = \sum_{T \subseteq S}                    \hhbar_{T}  .  
\end{equation}
For a subset $S \subseteq \{1, \ldots, n\}$
define the $\ab$-monomial $u_{S} = u_{1} u_{2} \cdots u_{n}$
by
$u_{i} = \av$ if $i \not\in S$ and
$u_{i} = \bv$ if $i \in S$.
The $\ab$-index of the
quasi-graded poset $(P,\rho,\zetabar)$ is then given by
$$  \Psi(P,\rho,\zetabar)
=
\sum_{S} \hhbar_{S} \cdot u_{S}  ,  $$
where the sum ranges over all subsets $S$.
Again, in the case when we take the weighted zeta function
to be the usual zeta function $\zeta$, the flag $\fbar$
and flag $\hhbar$-vectors correspond to the usual
flag $f$- and flag~$h$-vectors.

Using Lemma~\ref{lemma_well-known}
we have the next statement.
\begin{lemma}
For a quasi-graded poset $(P,\rho,\zetabar)$ of rank $n+1$
with minimal and maximal elements~$\hz$ and $\ho$,
the weighted M\"obius function $\mubar(\hz,\ho)$
is given by
$$
     \mubar(\hz,\ho) = (-1)^{n+1} \cdot \hhbar_{\{1,\ldots,n\}}  .
$$
\end{lemma}

We now give two recursions for the $\ab$-index.
\begin{proposition}
Let $(P,\rho,\zetabar)$ be a quasi-graded poset
where $\bar{\mu}$ is the inverse of $\zetabar$.
The following two recursions hold
for computing the $\ab$-index of an interval $[x,z]$:
\begin{align}
    \Psi([x,z],\rho,\zetabar) 
    & = 
    \zetabar(x,z)
    \cdot
    (\av-\bv)^{\rho(x,z)-1} \nonumber \\
    &   
    +
    \sum_{x < y < z}
    \Psi([x,y],\rho,\zetabar) 
    \cdot
    \bv
    \cdot
    \zetabar(y,z)
    \cdot
    (\av - \bv)^{\rho(y,z)-1}  ,
    \label{equation_straight_ab} \\
    \Psi([x,z],\rho,\zetabar)
    & = 
    - \mubar(x,z) \cdot (\av - \bv)^{\rho(x,z)-1} \nonumber \\
    &   
    -
    \sum_{x < y < z}
    \Psi([x,y],\rho,\zetabar)
    \cdot \av \cdot \mubar(y,z) \cdot (\av - \bv)^{\rho(y,z)-1}.
    \label{equation_Mobius_ab}
\end{align}
\label{proposition_recursions}
\end{proposition}
\begin{proof}
Using the chain definition of the $\ab$-index
and conditioning on the
largest element $y < z$
in the chain $c$, we obtain
the recursion~\eqref{equation_straight_ab}.
Multiplying equation~\eqref{equation_straight_ab}
by $\av-\bv$ on the right
and moving the term
$\Psi([x,z],\rho,\zetabar) \cdot \bv$
to the right-hand side, we obtain
\begin{equation}
  \Psi([x,z],\rho,\zetabar) 
  \cdot
  \av
  =
  \zetabar(x,z)
  \cdot
  (\av-\bv)^{\rho(x,z)}
  +
  \sum_{x < y \leq z}
  \Psi([x,y],\rho,\zetabar) 
  \cdot
  \bv
  \cdot
  \zetabar(y,z)
  \cdot
  (\av - \bv)^{\rho(y,z)}  .
  \label{equation_on_the_way_ab}
\end{equation}
Define three functions $f$, $g$ and $h$
in the incidence algebra of $P$ by
$$      f(x,y) = \begin{cases}
    \Psi([x,y],\rho,\zetabar) \cdot \av & \mbox{ if } x < y,\\
    1                     & \mbox{ if } x = y,
  \end{cases} 
\:\:\:\: \:\:\:\: 
g(x,y) = \begin{cases}
    \Psi([x,y],\rho,\zetabar) \cdot \bv & \mbox{ if } x < y,\\
    1                     & \mbox{ if } x = y,
  \end{cases}   $$
and 
$$
h(x,y) = \zetabar(x,y) \cdot (\av - \bv)^{\rho(x,y)}.
$$
Equation~\eqref{equation_on_the_way_ab}
can then be
written as $f = g \cdot h$ where the product is the convolution of the
incidence algebra. Observe that $h$ is invertible with its inverse
given by
$h^{-1}(x,y) = \mubar(x,y) \cdot (\av - \bv)^{\rho(x,y)}$.
By expanding the equivalent relation $g = f \cdot h^{-1}$,
we obtain
\begin{equation}
  \Psi([x,z],\rho,\zetabar) \cdot \bv
  =
  \mubar(x,z) \cdot (\av - \bv)^{\rho(x,z)}
  +
  \sum_{x < y \leq z}
  \Psi([x,y],\rho,\zetabar)
  \cdot \av \cdot \mubar(y,z) \cdot (\av - \bv)^{\rho(y,z)} .
\label{equation_expand}
\end{equation}
By moving the term $\Psi([x,z],\rho,\zetabar) \cdot \av$
to the left-hand side of 
equation~\eqref{equation_expand}
and canceling a factor of $\bv - \av$ on the right,
we obtain recursion~\eqref{equation_Mobius_ab}.
\end{proof}

Equation~\eqref{equation_Mobius_ab}
is an alternative recursion for the $\ab$-index
which may be viewed
as dual to~\eqref{equation_straight_ab}.
As a remark,  the two recursions
in Proposition~\ref{proposition_recursions}
contain the boundary condition
$\Psi([x,z],\rho,\zetabar)
=  \zetabar(x,z) \cdot (\av-\bv)^{\rho(x,z)-1}
= -\zetabar(x,z) \cdot (\av-\bv)^{\rho(x,z)-1}$
for $x$ covered by $z$.

We end this section with the essential result that the
$\ab$-index of a quasi-graded poset is a coalgebra homomorphism.
Define a coproduct $\Delta : \zab \longrightarrow \zab \tensor \zab$
by $\Delta(1) = 0$, 
$\Delta(\av) = \Delta(\bv) = 1 \tensor 1$
and for an $\ab$-monomial $u = u_{1} u_{2} \cdots u_{k}$
$$   \Delta(u)
=
\sum_{i=1}^{k} 
u_{1} \cdots u_{i-1}
\tensor
u_{i+1} \cdots u_{k}      . $$
The coproduct $\Delta$ extends to
$\zab$
by linearity.
It is straightforward to see that this coproduct is coassociative.
The coproduct $\Delta$ first appeared in~\cite{Ehrenborg_Readdy}.
\begin{theorem}
  Let $(P,\rho,\zetabar)$ be a quasi-graded poset.
  Then the following identity holds:
  $$
  \Delta(\Psi(P,\rho,\zetabar))
  =
  \sum_{\hz < x < \ho}
  \Psi([\hz,x],\rho,\zetabar)
  \tensor
  \Psi([x,\ho],\rho,\zetabar)  .
  $$
  \label{theorem_coalgebra_homomorphism}
\end{theorem}
The proof is the same as in the case of a graded poset~\cite{Ehrenborg_Readdy},
and hence is omitted. One way to formulate this result is to
say that the $\ab$-index is a coalgebra homomorphism
from the linear
span of quasi-graded posets to 
the 
algebra $\Zzz\langle \av, \bv \rangle$.

\section{Eulerian posets and Alexander duality}
\label{section_Euler}

We define a quasi-graded poset to be {\em Eulerian} if
for all pairs of elements $x \leq z$ we have that
\begin{equation}
  \sum_{x \leq y \leq z}
  (-1)^{\rho(x,y)} 
  \cdot
  \zetabar(x,y)
  \cdot
  \zetabar(y,z)
  =
  \delta_{x,z} .
  \label{equation_definition_Eulerian}
\end{equation}
In other words, the function 
$\mubar(x,y) = (-1)^{\rho(x,y)} \cdot \zetabar(x,y)$
is the inverse of $\zetabar(x,y)$ in the incidence algebra.
In the case $\zetabar(x,y) = \zeta(x,y)$, we refer to
relation~\eqref{equation_definition_Eulerian} as the
{\em classical Eulerian relation}.

If $x$ is covered by $z$, that is, there is
no element $y$ such that $x < y < z$, then the Eulerian
condition states that
either the rank difference $\rho(x,z)$ is odd
or the weighted zeta function $\zetabar(x,z)$ equals zero.
This statement will be generalized in
Proposition~\ref{proposition_odd_rank} below.
Similar to Lemma~3.1
in~\cite{Ehrenborg_Readdy_Eulerian_binomial}
and Exercise~3.174c in~\cite{Stanley_EC_1},
we will first need the following lemma.
\begin{lemma}
  Let $(P,\rho,\zetabar)$ be a 
  quasi-graded poset with $\hz$ and $\ho$ of odd rank $n$ such that
  every proper interval of $P$ is Eulerian. Then 
  $(P,\rho,\zetabar)$ is an Eulerian quasi-graded poset.
  \label{lemma_odd_rank}
\end{lemma}
\begin{proof}
  For any function $f$ in the incidence algebra $I(P)$
  satisfying $f(x,x) = 1$ for all $x$, we have
  $$
  f(\hz,\ho) + f^{-1}(\hz,\ho)
  =
  - \sum_{\hz < y < \ho} f(\hz,y) \cdot f^{-1}(y,\ho)
  =
  - \sum_{\hz < y < \ho} f^{-1}(\hz,y) \cdot f(y,\ho).
  $$
  Applying this relation to the case when $f= \zetabar$ gives
\begin{align*}
    \zetabar(\hz,\ho) + \zetabar^{-1}(\hz,\ho)
    & = 
    - \sum_{\hz < y < \ho}
    \zetabar(\hz,y) \cdot (-1)^{\rho(y,\ho)} \cdot \zetabar(y,\ho) \\
    & = 
    \sum_{\hz < y < \ho}
    (-1)^{\rho(\hz,y)} \cdot \zetabar(\hz,y) \cdot \zetabar(y,\ho) \\
    & = 
    - \zetabar(\hz,\ho) - \zetabar^{-1}(\hz,\ho) .
\end{align*}
  Here we are using that $\rho(\hz,y) + \rho(y,\ho) = n$ which is odd.
  Hence $\zetabar^{-1}(\hz,\ho) = (-1)^{n} \cdot \zetabar(\hz,\ho)$,
  that is, 
  $\zetabar^{-1}(x,z) = (-1)^{\rho(x,z)} \cdot \zetabar(x,z)$
  for all $x \leq z$ and thus the quasi-graded poset
  $(P,\rho,\zetabar)$ is Eulerian.
\end{proof}

\begin{proposition}
  Let $P$ be a poset with an order-preserving
  rank function $\rho$.
  Choose the values of $\zetabar(x,z)$ arbitrarily
  when the rank difference $\rho(x,z)$ is odd,
  and let
  \begin{equation}
    \zetabar(x,z)
    =
    -1/2
    \cdot
    \sum_{x < y < z}
    (-1)^{\rho(x,y)}
    \cdot
    \zetabar(x,y)
    \cdot
    \zetabar(y,z),   
    \label{equation_even_rank}
  \end{equation}
  for $x < z$ when $\rho(x,z)$ is even.
  Then $(P,\rho,\zetabar)$ is an Eulerian quasi-graded poset.
  \label{proposition_odd_rank}
\end{proposition}
\begin{proof}
  Lemma~\ref{lemma_odd_rank}
  guarantees that every odd interval is Eulerian
  and equation~\eqref{equation_even_rank}
  guarantees that every even interval is Eulerian.
\end{proof}

\begin{example}
{\rm
Let $C$ be the $n$ element chain
$C = \{x_{1} < x_{2} < \cdots < x_{n}\}$
with the rank function $\rho(x_{i}) = i$.
The number of intervals of odd rank in this poset
is given by
$\lfloor n/2 \rfloor \cdot \lceil n/2 \rceil$.
Hence we have
$\lfloor n/2 \rfloor \cdot \lceil n/2 \rceil$
degrees of freedom in choosing
a weighted zeta function $\zetabar$
in order for the quasi-graded  poset $(C,\rho,\zetabar)$
to be Eulerian.
}
\end{example}

The next lemma is well-known.
See for instance~\cite[Lemma~5.3]{Ehrenborg_Readdy_Sheffer}.
Since the incidence algebra of a finite poset with $n$ elements
is a subalgebra of all $n  \times n$ matrices,
this lemma is an instance of the identity
$(I+A)^{-1} = \sum_{k \geq 0} (-1)^{k} \cdot A^{k}$;
see~\cite[Section~3.6]{Stanley_EC_1}.
\begin{lemma}
  Let $P$ be a poset with minimal element $\hz$ and
  maximal element $\ho$.
  Let $f$ be a function in the incidence algebra $I(P)$
  such that $f(x,x) = 1$ for all $x$ in $P$.
  Then the inverse function of $f$ in the incidence algebra 
  $I(P)$ can be computed by
  $$    f^{-1}(\hz,\ho)
  =
  \sum_{\hz = x_{0} < x_{1} < \cdots < x_{k} = \ho}
  (-1)^{k}
  \cdot
  f(x_{0},x_{1})
  \cdot
  f(x_{1},x_{2})
  \cdots
  f(x_{k-1},x_{k})    .  $$
  \label{lemma_well-known}
\end{lemma}
Lemma~\ref{lemma_well-known} implies the next result.
\begin{lemma}
  Let $P$ be a poset with $\hz$ and $\ho$.
  Let $y$ be an element of $P$ such that $\hz < y < \ho$
  and
  let $Q$ be the subposet $P-\{y\}$.
  Suppose $f$ is a function in the incidence
  algebra of $P$ satisfying $f(x,x) = 1$
  for all $x \in P$.
  Then
  $$   (f|_{Q})^{-1}(\hz,\ho)
  =
  f^{-1}(\hz,\ho)
  -
  f^{-1}(\hz,y) \cdot f^{-1}(y,\ho)   .  $$
\end{lemma}
Iterating this lemma gives the following proposition.
\begin{proposition}
  Let $P$ be a poset with $\hz$ and $\ho$,
  and let $Q$ and $R$ be two subposets of $P$
  such that $Q \cup R = P$
  and $Q \cap R = \{\hz,\ho\}$.
  Assume $f$ is a function in the incidence algebra
  $I(P)$ satisfying $f(x,x) =1$ for all $x \in P$.
  Then
  $$   (f|_{Q})^{-1}(\hz,\ho)
  =
  \sum_{\onethingatopanother
    {\hz = y_{0} < y_{1} < \cdots < y_{k} = \ho}
    {y_{i} \in R}}
  (-1)^{k-1}
  \cdot
  f^{-1}(y_{0},y_{1})
  \cdot
  f^{-1}(y_{1},y_{2})
  \cdots
  f^{-1}(y_{k-1},y_{k})    .  $$
\end{proposition}
We now apply this result to Eulerian quasi-graded posets.
\begin{theorem}[Alexander duality for quasi-graded posets]
  Let $(P,\rho,\zetabar)$ be an Eulerian quasi-graded poset
  with $\hz$ and $\ho$ of rank $n+1$.
  Let $Q$ and $R$ be two subposets of $P$
  such that $Q \cup R = P$
  and $Q \cap R = \{\hz,\ho\}$.
  Then
  $$   (\zetabar|_{Q})^{-1}(\hz,\ho)
  =
  (-1)^{n}
  \cdot
  (\zetabar|_{R})^{-1}(\hz,\ho)  .
  $$
\label{theorem_Alexander}
\end{theorem}
\begin{proof}
Directly we have
\begin{align*}
    (\zetabar|_{Q})^{-1}(\hz,\ho)
    & = 
    \sum_{\onethingatopanother
      {\hz = y_{0} < y_{1} < \cdots < y_{k} = \ho}
      {y_{i} \in R}}
    (-1)^{k-1}
    \cdot
    \mubar(y_{0},y_{1})
    \cdots
    \mubar(y_{k-1},y_{k})  \\
    & = 
    \sum_{\onethingatopanother
      {\hz = y_{0} < y_{1} < \cdots < y_{k} = \ho}
      {y_{i} \in R}}
    (-1)^{k-1}
    \cdot
    (-1)^{\rho(y_{0},y_{1})} \cdot \zetabar(y_{0},y_{1})
    \cdots
    (-1)^{\rho(y_{k-1},y_{k})} \cdot \zetabar(y_{k-1},y_{k}) \\
    & =
    (-1)^{n}
    \cdot
    \sum_{\onethingatopanother
      {\hz = y_{0} < y_{1} < \cdots < y_{k} = \ho}
      {y_{i} \in R}}
    (-1)^{k}
    \cdot
    \zetabar(y_{0},y_{1})
    \cdots
    \zetabar(y_{k-1},y_{k}) \\
    & = 
    (-1)^{n}
    \cdot
    (\zetabar|_{R})^{-1}(\hz,\ho)  . \qedhere
\end{align*}
\end{proof}

Complementary pairs of posets can be constructed
using rank selection.

\begin{corollary}
  Let $(P,\rho,\zetabar)$ be an Eulerian quasi-graded
  poset of rank $n+1$
  with $\hz$ and $\ho$.
  Then the symmetric relation
  \begin{equation}
    \hhbar_{S} = \hhbar_{\overline{S}}
  \end{equation}
  holds for all subsets
  $S \subseteq \{1,2, \ldots, n\}$.
  \label{corollary_h_symmetry}
\end{corollary}
\begin{proof}
  Let $P_{S}$ be the $S$ rank-selected quasi-graded subposet
  $\{x \in P : \rho(x) \in S\} \cup \{\hz,\ho\}$.
  We have the following string of equalities:
\[
  \hhbar_{S}
  =
  (-1)^{|S|-1} \cdot (\zetabar|_{P_{S}})^{-1}(\hz,\ho)
  =
  (-1)^{n - |S| - 1}
  \cdot
  (\zetabar|_{P_{\overline{S}}})^{-1}(\hz,\ho)
  =
  (-1)^{|\overline{S}| - 1}
  \cdot
  (\zetabar|_{P_{\overline{S}}})^{-1}(\hz,\ho)
  =
  \hhbar_{\overline{S}}. \qedhere
\]
\end{proof}

\section{The $\cd$-index and quasi-graded posets}
\label{section_cd}

Bayer and Billera determined all the linear relations
which hold among the flag $f$-vector of (classical) Eulerian
posets, known
as the generalized Dehn--Sommerville relations~\cite{Bayer_Billera}.
Bayer and Klapper
showed that the space of flag $f$-vectors
of Eulerian posets has a natural basis
expressed by the $\cd$-index~\cite{Bayer_Klapper}.  
Since the degree of $\cv$ is $1$ and the degree of $\dv$ is $2$,
the dimension of the span of flag $f$-vectors of
Eulerian posets of rank $n+1$
is given by
the Fibonacci number $F_{n}$,
where $F_{0} = F_{1} = 1$ and
$F_{n} = F_{n-1} + F_{n-2}$.
Stanley later
gave a more elementary proof of the existence
of the $\cd$-index for Eulerian posets
and showed the coefficients are non-negative for
spherically-shellable posets~\cite{Stanley_d}.

\begin{theorem}[Bayer--Klapper]
For the face lattice of a polytope,
more generally, 
any graded Eulerian poset $P$,
its $\ab$-index
$\Psi(P)$ can be written uniquely as a polynomial
in the non-commutative variables
$\cv = \av + \bv$ and $\dv = \av\bv + \bv\av$
of degree one and two, respectively.
\end{theorem}

Generalizing the classical result of Bayer and Klapper
for graded Eulerian posets, we have the analogue for 
quasi-graded posets.
\begin{theorem}
For an Eulerian quasi-graded poset
$(P,\rho,\zetabar)$ its $\ab$-index
$\Psi(P,\rho,\zetabar)$ can be written uniquely
as a polynomial in the non-commutative
variables
$\cv = \av + \bv$ and $\dv = \av\bv + \bv\av$.
Furthermore, if the function~$\zetabar$ 
is integer-valued
then the $\cd$-index only has integer coefficients.
\label{theorem_cd}
\end{theorem}
\begin{proof}
Adding equations~\eqref{equation_straight_ab}
  and~\eqref{equation_Mobius_ab}
  in Proposition~\ref{proposition_recursions}
  and recalling that
  $\mubar(x,y) = (-1)^{\rho(x,y)} \cdot \zetabar(x,y)$,
  we obtain an expression for
  $2 \cdot \Psi([x,z],\rho,\zetabar)$
  in terms of the $\ab$-index
  of smaller intervals~$[x,y]$
  and
  the $\ab$-polynomials
  $(1 - (-1)^{k}) \cdot (\av-\bv)^{k-1}$
  and
  $(\bv - (-1)^{k} \cdot \av) \cdot (\av-\bv)^{k-1}$.
  Using Stanley's observation that these two $\ab$-polynomials
  can be written as $\cd$-polynomials~\cite{Stanley_d},
  and considering the two cases
  when $k$ is odd and when $k$ is even separately, the result follows.

If $\zetabar$ is integer-valued then it is clear that
the $\ab$-index only has integer coefficients. Finally,
the fact that the $\cd$-index only has integer coefficients
follows from~\cite[Lemma 3.5]{Ehrenborg_Readdy_Bruhat}.
\end{proof}

Theorem~\ref{theorem_cd} gives a different proof of
Corollary~\ref{corollary_h_symmetry} since
any $\cd$-polynomial when expressed in the variables $\av$
and $\bv$ is symmetric in $\av$ and $\bv$.

A different way to express the existence of the $\cd$-index is
as follows.
\begin{theorem}
The flag $\fbar$-vector of an Eulerian quasi-graded poset
of rank $n+1$ satisfies the generalized Dehn--Sommerville relations.
More precisely, for a subset $S \subseteq \{1, \ldots, n\}$
and $i,k \in S \cup \{0,n+1\}$
with
$i < k$ and $S \cap \{i+1, \ldots, k-1\} = \emptyset$,
the following relation holds:
\begin{equation}
   \sum_{j=i}^{k} (-1)^{j} \cdot \fbar_{S \cup \{j\}} = 0.
\label{equation_generalized_Dehn--Sommerville}
\end{equation}
\label{theorem_generalized_Dehn--Sommerville}
\end{theorem}
The relations in Theorem~\ref{theorem_generalized_Dehn--Sommerville}
reduce to the classical Dehn--Sommerville relations~\cite{Dehn,Sommerville}
for Eulerian simplicial posets.
See~\cite{Mulmuley} for results related to the classical 
Dehn--Sommerville relations.

The next result asserts that if the $\cd$-index exists for every interval
in a quasi-graded poset then the poset itself is Eulerian.
\begin{proposition}
Let $(P,\rho,\zetabar)$ be a quasi-graded poset such
that the $\ab$-index of every interval can be expressed
in terms of $\cv$ and $\dv$.  Then the
quasi-graded poset $(P,\rho,\zetabar)$ is Eulerian.
\label{proposition_cd_implies_Eulerian}
\end{proposition}
\begin{proof}
Let $[x,y]$ be an interval of rank $k+1$.
For an $\ab$-monomial $m$ and
$q(\av,\bv)$ a non-commutative polynomial in the
variables $\av$ and $\bv$, let
$[m] q(a,b)$
denote the coefficient of $m$ in $q(\av,\bv)$. 
We have
\begin{align*}
\mubar(x,y)
  & = 
(-1)^{k+1} \cdot \hhbar_{\{1, \ldots, k\}}([x,y])  \\
  & = 
(-1)^{k+1} \cdot [\bv^{k}] \Psi([x,y]) \\
  & = 
(-1)^{k+1} \cdot [\av^{k}] \Psi([x,y]) \\
  & = 
(-1)^{k+1} \cdot \hhbar_{\emptyset}([x,y])  \\
  & = 
(-1)^{k+1} \cdot \zetabar(x,y) ,
\end{align*}
where the third equality follows from
expanding the $\cd$-polynomial $\Psi([x,y])$
in terms of $\av$'s and~$\bv$'s.
The resulting identity is 
the Eulerian relation~\eqref{equation_definition_Eulerian}.
\end{proof}

\section{Poset operations}
\label{section_poset_operations}

In this section we extend some standard poset operations to
quasi-graded posets.

Given a quasi-graded poset $(P,\rho_{P},\zetabar_{P})$, its
{\em dual} is the quasi-graded poset
$(P^*,\rho_{P^*},\zetabar_{P^*})$,
where the partial order satisfies
$x \leq_{P^*} y$
if $y \leq_{P} x$,
the weighted zeta function is given by
$\zetabar_{P^*}(x,y) = 
\zetabar_{P}(y,x)$,
and
the rank function is
$\rho_{P^*}(x) = -\rho_{P}(x)$.
When the poset $P$ has a minimal and maximal element,
we prefer to use
the rank function
$\rho_{P^*}(x) = \rho_{P}(\ho)-\rho_{P}(x)$
in order that the minimal element in $P^{*}$ has rank $0$.

Let $*$ also denote the involution on $\zab$ that
reverses each monomial. Then for a graded quasi-graded
poset $P$ we have that
$\Psi(P^*,\rho_{P^*},\zetabar_{P^*})
  =
 \Psi(P,\rho_{P},\zetabar_{P})^{*}$.

The {\em Cartesian product}
of two quasi-graded posets
$(P,\rho_{P},\zetabar_{P})$
and $(Q,\rho_{Q},\zetabar_{Q})$
is the triple
$(P \times Q, \rho,\zetabar)$
where
the rank function is given by
the sum
$\rho((x,y)) = \rho_{P}(x) + \rho_{Q}(y)$
and the
weighted zeta function is
the product
$\zetabar((x,y),(z,w))
=
\zetabar_{P}(x,z) \cdot \zetabar_{Q}(y,w)$.

It is straightforward to verify the following result.

\begin{proposition}
  If two quasi-graded posets
  $(P,\rho_{P},\zetabar_{P})$
  and $(Q,\rho_{Q},\zetabar_{Q})$
  are both Eulerian then so is their Cartesian product.
\end{proposition}

One important case is
the Cartesian product with the Boolean algebra $B_{1}$.
The geometric motivation  is that
taking the Cartesian product of
the  face lattice of a polytope with $B_{1}$ corresponds to
taking the pyramid of the polytope.
To do this, define the derivation 
$G: \zab \rightarrow \zab$
by 
$G(\av) = \bv\av$ and
$G(\bv) = \av\bv$.
Observe that $G$ restricts to a derivation
on $\cd$-polynomials by
$G(\cv) = \dv$ and
$G(\dv) = \cd$.
Also define the operator $\Pyr$ by
$\Pyr(w) = w \cdot \cv + G(w)$.
These two operators first appeared in~\cite{Ehrenborg_Readdy}.
\begin{proposition}
  Let $P$ be a quasi-graded poset
  with $\hz$ and $\ho$. 
  Then the $\ab$-index of the Cartesian product 
  of~$P$ with the Boolean algebra $B_{1}$
  is given by
  \begin{align}
    \Psi(P \times B_{1})
    & = 
    \bv \cdot \Psi(P)
    +
    \Psi(P) \cdot \av
    +
    \sum_{\onethingatopanother{x \in P}{\hz < x < \ho}}
    \Psi([\hz,x]) \cdot \av \bv \cdot \Psi([x,\ho]) ,
    \label{equation_pyramid_one} \\
    \Psi(P \times B_{1})
    & = 
    \Pyr(\Psi(P)) .
    \label{equation_pyramid_two}
  \end{align}
\label{proposition_pyramid}
\end{proposition}
\begin{proof}
  The proof follows the same outline as
  in~\cite[Proposition 4.2]{Ehrenborg_Readdy}. 
  Consider a chain
  $c = \{ (\hz,\hz) = (x_{0},y_{0}) < (x_{1},y_{1}) 
  < \cdots < (x_{k},y_{k})  = (\ho,\ho) \}$
  in $P \times B_{1}$.
  Let $i$ be the smallest index such that $y_{i} = \ho$
  and let $x = x_{i}$.
  Assume that $\hz < x < \ho$. Notice that
  the element $(x,\hz)$ may or may not be in the chain~$c$.
  Let $c^{\prime}$ denote the chain $c - \{ (x,\hz) \}$
  and
  let $c^{\prime\prime}$ denote the chain $c \cup \{ (x,\hz) \}$.
  The $\ab$-weights of the chains are respectively given by
  \begin{align*}
    w(c^{\prime})
    & = 
    w_{[\hz,x]}( c_{1} )
    \cdot
    (\av - \bv)
    \cdot
    \bv
    \cdot
    w_{[x,\ho]}( c_{2} ) , \\
    w(c^{\prime\prime})
    & = 
    w_{[\hz,x]}( c_{1} )
    \cdot
    \bv
    \cdot
    \bv
    \cdot
    w_{[x,\ho]}( c_{2} ) ,
  \end{align*}
  where $c_{1}$ and $c_{2}$ are the chains obtained by
  restricting the chain $c$ to
  the interval $[(\hz,\hz),(x,\hz)] \cong [\hz,x]$,
  respectively 
  $[(x,\ho),(\ho,\ho)] \cong [x,\ho]$.
  Let $(w,\hz)$ be the element preceding  $(x,\hz)$
  in the chain $c^{\prime\prime}$.
  Since the weighted zeta function is multiplicative
  and 
  $\zetabar_{B_{1}}(\hz,\ho) = 1$
  we have
  $\zetabar((w,\hz),(x,\ho))
  =
  \zetabar((w,\hz),(x,\hz))$.
  Thus we conclude that the two chains
  $c^{\prime}$ and $c^{\prime\prime}$
  have the same $\zetabar$-weight
  given by the product
  $\zetabar(c^{\prime})
  =
  \zetabar(c^{\prime\prime})
  =
  \zetabar(c_{1})
  \cdot
  \zetabar(c_{2})$.
  Summing over all chains $c_{1}$ and $c_{2}$
  and over all elements $x$, we obtain
  the summation expression in~\eqref{equation_pyramid_one}.
  The other two terms on the right-hand side
  of~\eqref{equation_pyramid_one}
  follow from the two
  cases $x = \hz$ and $x = \ho$.

The proof of the second identity
follows the same outline as the proof of Theorem~5.2
in~\cite{Ehrenborg_Readdy},
using that the $\ab$-index is
a coalgebra homomorphism.
See Theorem~\ref{theorem_coalgebra_homomorphism}.
\end{proof}

The {\em Stanley product} of two
quasi-graded posets
$(P,\rho_{P},\zetabar_{P})$
and $(Q,\rho_{Q},\zetabar_{Q})$
with minimal and maximal elements
is the triple
$(R,\rho,\zetabar)$.
The poset $R$
is described by 
$R = (P - \{\ho\}) \cup (Q - \{\hz\})$
with
the partial order given by
$x \leq_{P*Q} y$
if
(i)   $x \leq_{P} y$ and $x,y \in P-\{\ho\}$;
(ii)  $x \in P-\{\ho\}$ and $y \in Q-\{\hz\}$; or
(iii) $x \leq_{Q} y$ and $x,y \in Q-\{\hz\}$.
The rank function $\rho$ is given by
$\rho(x) = \rho_{P}(x)$ if $x \in P - \{\ho\}$
and
$\rho(x) = \rho_{Q}(x) + \rho_{P}(\hz,\ho) - 1$
if $x \in Q - \{\hz\}$.
Finally, the weighted zeta function is given by
$\zetabar(x,y) = \zetabar_{P}(x,y)$
if $x,y \in P - \{\ho\}$;
$\zetabar(x,y) = \zetabar_{Q}(x,y)$
if $x,y \in Q - \{\hz\}$; or
$\zetabar(x,y) = \zetabar_{P}(x,\ho) \cdot \zetabar_{Q}(\hz,y)$
if $x \in P - \{\ho\}$ and $y \in Q - \{\hz\}$.

If 
$c = \{\hz = x_{0} < x_{1} < \cdots < x_{p} < y_{1} < y_{2} < \cdots < y_{q} = \ho\}$
is a chain in the product $P*Q$
then  each $x_{i} \in P$
and each $y_{j} \in Q$. Let $c_{P}$ and $c_{Q}$ denote
the two restricted chains 
$c_{P} = \{\hz = x_{0} < x_{1} < \cdots < x_{p} < \ho\}$,
respectively
$c_{Q} = \{\hz < y_{1} < y_{2} < \cdots < y_{q} = \ho\}$.
Note that the zeta weight of the chains is multiplicative,
that is, $\zetabar(c) = \zetabar_{P}(c_{P}) \cdot \zetabar_{Q}(c_{Q})$.
Thus summing over all chains yields the next result.
\begin{lemma}
  The $\ab$-index is multiplicative with
  respect to
  the Stanley product of quasi-graded posets, that is,
  for two quasi-graded posets $P$ and $Q$:
  $$   \Psi(P*Q) = \Psi(P) \cdot \Psi(Q)   .  $$
\label{lemma_Stanley_product}
\end{lemma}

Recall the {\em up set} $U(x)$ of
an element $x$ from a poset $P$ is the set 
$U(x) = \{ v \in P \:\: : \: \: x < v\}$.

\begin{lemma}
Let $(P,\rho,\zetabar)$ be a quasi-graded poset
with two elements $x$ and $y$
such that 
$\rho(x) = \rho(y)$,
$U(x) = U(y)$
and $\zetabar(x,v) = \zetabar(y,v)$
for all elements $v$ in the joint up set.
Let $(Q,\rho_{Q},\zetabar_{Q})$ be the quasi-graded poset
defined by
\vspace*{-4mm}
\begin{enumerate}
\item[(i)]
$Q = P - \{x,y\} \cup \{w\}$,

\item[(ii)]
$Q$ inherits the order relation from $P$
with the new relations
$u < x$ or $u < y$ implies $u <_{Q} w$,
and
$x < v$ (hence $y < v$) implies $w <_{Q} v$.

\item[(iii)]
$Q$ inherits the rank function from $P$
with the new value $\rho_{Q}(w) = \rho(x)$,

\item[(iv)]   
$Q$ inherits the weighted zeta function from $P$
with the new values
$\zetabar_{Q}(u,w) = \zetabar(u,x) + \zetabar(u,y)$
and
$\zetabar_{Q}(w,v) = \zetabar(x,v)$.
\end{enumerate}
\vspace*{-4mm}
Then the $\ab$-indexes of
$(P,\rho,\zetabar)$ 
and
$(Q,\rho_{Q},\zetabar_{Q})$
are equal, that is,
$\Psi(P,\rho,\zetabar) = \Psi(Q,\rho_{Q},\zetabar_{Q})$.
\label{lemma_x_y_w}
\end{lemma}
The proof follows by straightforward chain enumeration.

The other operation we need for quasi-graded posets
is zipping. This was originally developed by Reading
for Eulerian posets~\cite{Reading}. 
For our purposes we are considering the dual
situation with respect to Reading's original notion.
\begin{definition}
A {\em zipper} in a quasi-graded poset $(P,\rho,\zetabar)$
consists of three elements $x$, $y$ and $z$ such that
\vspace*{-4mm}
\begin{enumerate}
\item[(i)]
$\rho(x) = \rho(y) = \rho(z)+1$,

\item[(ii)]  
the element $z$ is only covered by $x$ and $y$,

\item[(iii)]   
for all $v > z$ with $\rho(v) > \rho(x)$ we have
              $\zetabar(x,v) = \zetabar(y,v) = \zetabar(z,v)$,

\item[(iv)]    
$\zetabar(z,x) = \zetabar(z,y) = 1$.
\end{enumerate}
\label{definition_zipper}
\end{definition}
As a consequence of condition~($iii$),
the elements $x$ and $y$ in a zipper have the same up set.

The next step is to zip the zipper.
\begin{definition}
Let $(P,\rho,\zetabar)$ be a quasi-graded poset with zipper $x$, $y$ and $z$.
The {\em zipped} quasi-graded poset~$Q$ consists
of the elements $Q = P - \{x,y,z\} \cup \{w\}$ such that
the poset $Q$ inherits the partial order relation from~$P$
together with the following new relations:
if $u <_{P} x$ or $u <_{P} y$ then
$u <_{Q} w$ and 
if $x <_{P} v$ or $y <_{P} v$
then $w <_{Q} v$.
The rank function of $Q$ is also inherited from $P$
with the new value $\rho_{Q}(w) = \rho(x)$.
Finally, the weighted zeta function has the new values
\begin{equation*}
     \zetabar_{Q}(w,v) = \zetabar(x,v)
     \:\:\:\mbox{     and     }  \:\:\:
     \zetabar_{Q}(u,w) = \zetabar(u,x) + \zetabar(u,y) - \zetabar(u,z).
\end{equation*}
\end{definition}
Observe that we use the fact
$\zetabar(u,v) = 0$ if $u \not\leq v$.
Hence the last relation implies that
$\zetabar_{Q}(u,w) = \zetabar(u,x)$ if $u < x$ and $u \not< y$,
and vice-versa, with the roles of $x$ and $y$ exchanged.
Similarly,
$\zetabar_{Q}(u,w) = \zetabar(u,x) + \zetabar(u,y)$
if $u < x$, $u < y$ but $u \not< z$.

\begin{example}
{\rm
As a remark, we do not need (the dual of) Reading's original condition
that $x \meet y = z$ in the definition of a zipper since we
do not require the weighted zeta function
of a zipped quasi-graded poset $\zetabar_{Q}$
to always be equal to $1$. 
As a concrete example, consider
the face lattice of the $2$-gon
with vertices $v_{1}$ and $v_{2}$
and edges $e_{1}$ and $e_{2}$.
This poset is the rank $3$ butterfly poset
and it is Eulerian with the weighted
zeta function $\zetabar$ equal to the
zeta function $\zeta$.
The triple $e_{1}$, $e_{2}$ and~$v_{1}$ is a zipper.
After zipping we obtain 
the length $3$ chain 
$Q = \{\hz \coveredby v_{2} \coveredby w \coveredby \ho\}$.
The weighted zeta function for~$Q$ 
is given by $\zetabar_Q(v_{2},w) = 2$
and $1$ everywhere else.
It is easily checked that the resulting poset is Eulerian.
}
\end{example}

\begin{proposition}
Let $(P,\rho,\zetabar)$ be a quasi-graded poset
with minimal and maximal elements.
Assume that $P$ has a zipper $x$, $y$ and $z$.
Then the $\ab$-indexes of the zipped poset $(Q,\rho_{Q},\zetabar_{Q})$
and of the interval $([\hz,w],\rho_{Q},\zetabar_{Q})$
are given by
\begin{align}
   \Psi(Q)
  & = 
   \Psi(P)
      -
   \Psi([\hz,z]) \cdot \dv \cdot \Psi([z,\ho])  , 
\label{equation_zipping_part_1} \\
   \Psi([\hz,w])
  & = 
   \Psi([\hz,x])
     +
   \Psi([\hz,y])
     -
   \Psi([\hz,z]) \cdot \cv .
\label{equation_zipping_part_2}
\end{align}
\label{proposition_zipping}
\end{proposition}
\begin{proof}
We begin by summing the weights of chains
from $P$ which go through
the elements $x$, $y$ and~$z$.
The chains that contain $z$
(and possibly $x$, $y$ or neither)
are enumerated by
\begin{equation}
    \Psi([\hz,z]) \cdot (2 \cdot \bv\bv + \bv (\av-\bv)) \cdot \Psi([z,\ho]) .
\label{equation_zip_1}
\end{equation}
Here we use that $\zetabar(x,v) = \zetabar(y,v) = \zetabar(z,v)$.
The chains in $P$ that contain $x$ but not $z$ are enumerated by
\begin{equation}
     \zetabar(\hz,x) \cdot (\av-\bv)^{\rho(x)-1}
                   \cdot \bv
                   \cdot \Psi([x,\ho])  
 +
    \sum_{\onethingatopanother{\hz < u < x}{u \neq z}}
     \Psi([\hz,u]) \cdot \bv
                   \cdot \zetabar(u,x) \cdot (\av-\bv)^{\rho(u,x)-1}
                   \cdot \bv
                   \cdot \Psi([x,\ho])  .
\label{equation_zip_2}
\end{equation}
Similarly, the chains in $P$ through $y$ 
which do not contain $z$ are enumerated by
\begin{equation}
     \zetabar(\hz,y) \cdot (\av-\bv)^{\rho(y)-1}
                   \cdot \bv
                   \cdot \Psi([y,\ho])  
 +
    \sum_{\onethingatopanother{\hz < u < y}{u \neq z}}
     \Psi([\hz,u]) \cdot \bv
                   \cdot \zetabar(u,y) \cdot (\av-\bv)^{\rho(u,y)-1}
                   \cdot \bv
                   \cdot \Psi([y,\ho])  .
\label{equation_zip_3}
\end{equation}
Finally, we enumerate the chains in the poset $Q$
through the new element $w$:
\begin{equation}
     \zetabar_{Q}(\hz,w) \cdot (\av-\bv)^{\rho_{Q}(w)-1}
                   \cdot \bv
                   \cdot \Psi([w,\ho])  
 +
    \sum_{\hz < u < w}
     \Psi([\hz,u]) \cdot \bv
                   \cdot \zetabar_{Q}(u,w) \cdot (\av-\bv)^{\rho_{Q}(u,w)-1}
                   \cdot \bv
                   \cdot \Psi([w,\ho])  .
\label{equation_zip_4}
\end{equation}
Adding~\eqref{equation_zip_2}
and~\eqref{equation_zip_3}
and
subtracting~\eqref{equation_zip_4} gives
\begin{align}
  &   
     \zetabar(\hz,z) \cdot (\av-\bv)^{\rho(z)}
                   \cdot \bv
                   \cdot \Psi([x,\ho])  
 +
    \sum_{\hz < u < z}
     \Psi([\hz,u]) \cdot \bv
                   \cdot \zetabar(u,z) \cdot (\av-\bv)^{\rho(u,z)}
                   \cdot \bv
                   \cdot \Psi([x,\ho])  
\nonumber \\
  & = 
     \Psi([\hz,z]) \cdot (\av-\bv)
                   \cdot \bv
                   \cdot \Psi([x,\ho]) .
\label{equation_zip_5}
\end{align}
The chains that jump from $u$ to $x$ where $u \not< y$
cancel with the corresponding chains jumping from~$u$ to~$w$.
A similar cancellation occurs with the roles of $x$ and $y$ exchanged.
Hence the difference $\Psi(P) - \Psi(Q)$ is given by
the sum of~\eqref{equation_zip_1}
and~\eqref{equation_zip_5}.

For the interval $[\hz,w]$
in the poset $Q$
we have
\begin{align*}
\Psi([\hz,w])
  & = 
\zetabar_{Q}(\hz,w) \cdot (\av-\bv)^{\rho(w)-1}
   +
\sum_{\hz < u < w}
   \Psi([\hz,u])
       \cdot
   \bv
       \cdot
   \zetabar_{Q}(u,w) \cdot (\av-\bv)^{\rho(u,w)-1} \\
  & = 
\zetabar(\hz,x) \cdot (\av-\bv)^{\rho(x)-1}
   +
\sum_{\onethingatopanother{\hz < u < x}{u \neq z}}
   \Psi([\hz,u])
       \cdot
   \bv
       \cdot
   \zetabar(u,x) \cdot (\av-\bv)^{\rho(u,x)-1} \\
  &   
   +
\zetabar(\hz,y) \cdot (\av-\bv)^{\rho(y)-1}
   +
\sum_{\onethingatopanother{\hz < u < y}{u \neq z}}
   \Psi([\hz,u])
       \cdot
   \bv
       \cdot
   \zetabar(u,y) \cdot (\av-\bv)^{\rho(u,y)-1} \\
  &   
   -
\zetabar(\hz,z) \cdot (\av-\bv)^{\rho(z)}
   -
\sum_{\hz < u < z}
   \Psi([\hz,u])
       \cdot
   \bv
       \cdot
   \zetabar(u,z) \cdot (\av-\bv)^{\rho(u,z)} \\
  & = 
\Psi([\hz,x])
   -
   \Psi([\hz,z])
       \cdot
   \bv
   +
\Psi([\hz,y])
   -
   \Psi([\hz,z])
       \cdot
   \bv
   -
\Psi([\hz,z]) \cdot (\av-\bv) ,
\end{align*}
where in the second and third equality the computations are taking
place in the original poset $P$.
This is the desired expression.
\end{proof}

\begin{theorem}
For quasi-graded posets the zipping operation
preserves the Eulerian property.
\label{theorem_zipping_preserves_Eulerian}
\end{theorem}
\begin{proof}
Consider an interval $[u,v]$ in the zipped poset $Q$,
where $P$ is the original quasi-graded poset with
zipper $x$, $y$ and $z$.
First we claim that this interval has a $\cd$-index.
There are four cases to consider.
(i)
If $w$ does not belong to 
the interval $[u,v]$ then
$[u,v]$ is the same as in the original poset~$P$
and hence is Eulerian.
(ii)
If $u=w$ then $[w,v]$ is isomorphic to $[x,v]$
and is Eulerian.
(iii)
If $v=w$ then the interval $[u,w]$ has a $\cd$-index by
equation~\eqref{equation_zipping_part_2}.
(iv)
Finally, if $u < w < v$ then the interval $[u,v]$ has a $\cd$-index by
equation~\eqref{equation_zipping_part_1}.
In each case the interval has a $\cd$-index and
by
Proposition~\ref{proposition_cd_implies_Eulerian}
we conclude that the poset $Q$ is Eulerian.
\end{proof}

\section{Whitney stratified sets}
\label{sec-stratified-sets}

We begin with a modest example.
\begin{example}
{\rm
Consider the non-regular $CW$~complex $\Omega$
consisting of one vertex $v$, one edge $e$ and
one $2$-dimensional cell $c$ such that the boundary of $c$
is the union $v \cup e$,
that is, boundary of the complex $\Omega$ is a one-gon.
Its face poset is the four element chain
${\mathscr F}(\Omega) = \{\hz < v < e < c\}$.
This is not an Eulerian poset.
The classical definition of the $\ab$-index,
in other words, using $\zetabar(x,y) = 1$ for all $x \leq y$,
yields that the $\ab$-index of~$\Omega$ is $\av^{2}$.
Note that $\av^{2}$ cannot be written in terms
of $\cv$ and~$\dv$.

Observe that the edge $e$ is attached to the vertex $v$
twice. Hence it is natural to change the value
of $\zetabar(v,e)$ to be $2$
and to keep the remaining
values of $\zetabar(x,y)$ to be $1$.
The face poset ${\mathscr F}(\Omega)$ is now Eulerian,
its $\ab$-index is given by
$\zetabar(\Omega) = \av^{2} + \bv^{2}$
and hence its $\cd$-index is
$\zetabar(\Omega) = \cv^{2} - \dv$.
}
\label{example_one-gon}
\end{example}

The motivation for the value $2$
in Example~\ref{example_one-gon}
is best expressed in terms
of the Euler characteristic of the link.
The link of the vertex
$v$ in the edge $e$ is two points whose
Euler characteristic is~$2$.
In order to view this example
in the right topological setting,
we review the notion of a Whitney stratification.
For more details,
see~\cite{du_Plessis_Wall},
\cite{Gibson_Wirthmuller_du_Plessis_Looijenga},
\cite[Part I \S 1.2]{Goresky_MacPherson},
and 
\cite{Mather}.

A subset $S$ of a topological space $M$
is {\em locally closed} if $S$ is
a relatively open subset of its closure~$\overline{S}$.
Equivalently, for any  point $x \in S$ there exists a neighborhood
$U_x \subseteq S$ such that the closure
$\overline{U_x} \subseteq S$ is closed in $M$.  
Another way to phrase this is a subset $S \subset M$
is locally closed if and only if it is the intersection of an
open subset and a closed subset of $M$.

\begin{definition}
Let $W$ be a closed subset of a smooth manifold $M$ which has been
decomposed into a finite union of locally closed subsets
called strata:
$$
  W = \bigcup_{X \in {\mathcal P}} X.
$$
Furthermore suppose this decomposition satisfies the {\em condition
of the frontier}:
$$
      X \cap \overline{Y} \neq \emptyset
  \Longleftrightarrow
      X \subseteq \overline{Y}.
$$
This implies the closure of each stratum is a union of strata,
and it provides the index set $\mathcal P$ with the partial ordering:
$$
      X \subseteq \overline{Y}
  \Longleftrightarrow
      X \leq_{\mathcal P} Y.
$$
This decomposition of $W$ is a {\em Whitney stratification}
if  \vspace*{-4mm}
\begin{enumerate}
\item
Each $X \in {\mathcal P}$
is a (locally closed, not necessarily connected)
smooth submanifold of $M$.

\item
If $X <_{\mathcal P} Y$ then Whitney's conditions
(A) and (B) hold:
Suppose $y_{i} \in Y$ is a sequence of points converging to some
$x \in X$ and that $x_{i} \in X$ converges to $x$.
Also assume that (with respect to some local
coordinate system on the manifold $M$) the secant lines
$\ell_{i} = \overline{x_{i} y_{i}}$ converge to some
limiting line $\ell$ and the tangent planes
$T_{y_{i}} Y$ converge to some limiting plane
$\tau$.
Then the following inclusions hold:
$$
    \text{ (A) } \: T_{x} X \subseteq \tau
\:\:\:\:\:\:\:\:
    \text{  and  }
\:\:\:\:\:\:\:\:
    \text{ (B) } \: \ell \subseteq \tau.
$$
\end{enumerate}
\end{definition}
\begin{remark}
{\rm
For convenience we will henceforth also assume that $W$ is pure dimensional,
meaning that if $\dim(W) = n$ then the union of the $n$-dimensional strata of
$W$ forms a dense subset of $W$.  Strata of dimension less than $n$
are referred to as {\em singular strata}.
}
\end{remark}

\paragraph{Local structure of stratified sets.}
Lemmas~\ref{lemma_link_well_defined} and~\ref{lemma_local_product}
below are ``standard'' but the proofs (which we omit) involve the full
strength of Thom's {\em first isotopy lemma};
see~\cite{du_Plessis_Wall,Gibson_Wirthmuller_du_Plessis_Looijenga,
Goresky_MacPherson,Mather, Mather2,Thom}.

Let $A, A' \subset M$ be smooth submanifolds.  They are said to be
{\em transverse in $M$ at a point} $x \in A \cap A'$ if
$T_xA + T_xA' = T_x M$.  The submanifolds $A, A'$ are said to be transverse (in $M$) if they are transverse at every point of their
intersection.

Two Whitney stratified subsets $W,W' \subset M$ are said to be
{\em transverse at a point} $x \in W \cap W'$ if the stratum $A$ of
$W$ that contains $x$ is transverse at the point $x$ to the stratum
$A'$ of $W'$ that contains $x$.  The Whitney stratified sets $W,W'$
are said to be {\em transverse} if they are transverse at every point,
that is, if every stratum $A$ of $W$ is transverse to every stratum
$A'$ of $W'$.  In this case, the intersection $W \cap W'$ is Whitney
stratified with strata of the form $A \cap A'$ where $A$ is a stratum
of $W$ and $A'$ is a stratum of $W'$.

Transversality is an {\em open}
condition:  if $A \le_{\mathcal P} B$ and if $A' \le_{\mathcal P'} B'$
are strata of $W,W'$ respectively and if $A$ is transverse to $A'$ at
a point $x \in A \cap A'$ then there exists a neighborhood
$U\subset M$ of the point $x$ such that  $W \cap U$ is transverse to
$W' \cap U$, in other words, 
the following four conditions hold:
(i)~$A$ is transverse to~$A'$ at every point $y\in A \cap A' \cap U$,
(ii)~$A$ is transverse to~$B'$ at every point $y\in A \cap B' \cap U$,
(iii)~$B$ is transverse to~$A'$ at every point $y\in B \cap A' \cap U$ and
(iv)~$B$ is transverse to~$B'$ at every point $y\in B \cap B' \cap U$.
Transversality is also a {\em dense} condition:  two compact Whitney
stratified subsets $W, W' \subset \mathbb R^n$ of Euclidean space may
be made transverse by moving one of them by a translation, say,
$W' \rightarrow W'+a$, where the vector $a \in \mathbb R^n$ may be
chosen to be arbitrarily small.

Let $W$ be a Whitney stratified closed subset of a smooth manifold $M$.
Let $X$ be a stratum of~$W$ and let $x\in X$. Let $N_x \subset M$ be a {\em normal
slice} to $X$ at $x$, that is, a smooth
submanifold of~$M$ that is transverse to $X,$ with $N_x \cap X = \{x\}$ (so
that $\dim(N_x) + \dim(X) = \dim(M)$).  Let $B_{\epsilon}(x)$ be a closed
ball (with respect to some local coordinate system for $M$) of radius $\epsilon$
centered at $x$.  If $\epsilon$ is chosen sufficiently small then
\begin{enumerate}
\item the boundary $\partial B_{\epsilon}(x)$ is transverse to $N_x$,
\item the intersection $\partial B_{\epsilon}(x) \cap N_x$ is transverse to
every stratum of $W$.
\end{enumerate}
Consequently the intersection
\[ \link_W(X,x):= N_x \cap \partial B_{\epsilon}(x) \cap W\]
is Whitney stratified by its intersection with the strata of $W$.  The link is well-defined in the following sense.
\begin{lemma}
Assume the stratum $X$ is connected.  Let $x'\in X$,
let $N'_{x'}$ be a choice of normal slice to~$X$ at the point $x'$ and let
$B'_{\epsilon'}(x')$ be a choice of closed ball centered at $x'$.  Then for
$\epsilon, \epsilon'$ sufficiently small the intersections
\[
 \link_W(X,x):= N_x \cap \partial B_{\epsilon}(x) \cap W \ \text{ and }\
\link_W(X,x'):= N'_{x'} \cap \partial B'_{\epsilon'}(x') \cap W\]
are homeomorphic by a homeomorphism $\psi$ that is smooth on each stratum
such that $\psi(x) = x'$.
\label{lemma_link_well_defined}
\end{lemma}
We may therefore refer to ``the'' link of the stratum $X$ by
choosing a point $x\in X$ and writing $\link_W(X) = \link_W(X,x)$.
If $X \subset \overline{Y}$ are strata of $W$ and
if $x \in X$ then we write
$$
     \link_Y(X) = N_x \cap \partial B_{\epsilon}(x) \cap Y
$$
for the intersection of the link of $X$ with the stratum $Y,$ and
similarly for $\link_{\overline{Y}}(X)$.  These intersections are also
independent of the same choices as described above.

The link is preserved under transverse intersection
in the following sense.
Let $P\subset M$ be a smooth
submanifold that is transverse to every stratum of $W$.  Let $X$ be a stratum of
$W$ and let $x \in X\cap P$.  Then there is a stratum-preserving
homeomorphism that is smooth on each stratum,
\begin{equation}\label{equation_link_intersection}
\link_{P\cap W}(P\cap X,x)\cong \link_W(X,x).
\end{equation}
This follows from Lemma~\ref{lemma_link_well_defined}
by choosing the normal slice
$N_x$ to be contained in $P$.  With this choice, $\link_{P\cap W}(P\cap X,x) =
\link_W(X,x)$.

\begin{remark}
{\rm
The ``top'' stratum, or largest stratum of
a Whitney stratified set $W$ need not be connected.
Normally, one requires that the singular strata of
a Whitney stratified set $W$ should be connected.
However it is possible to allow disconnected singular strata provided
\begin{itemize}
\item[(i)]
that the condition of the frontier continues to hold, and
\item[(ii)]
that the links $L_1$, $L_2$ of any two
connected components $A_1$, $A_2$ of a singular stratum
$A$ are isomorphic,
meaning that there is a stratum preserving
homeomorphism $h_{12}:L_1 \longto L_2$ that is smooth on each stratum.
\end{itemize}
(Here, the stratification of $L_i$ is the natural one given by the
intersection of $L_i$ with the strata of~$W$, for $i=1,2.$)
}
\end{remark}

\begin{lemma}
Let $W\subset M$ be a Whitney stratified
closed subset of a smooth manifold $M$.
Let $X$ be a stratum of $W$ and let $x\in X$.
Then there exists a basis $\mathcal B_x$ for
the neighborhoods of $x$ in $W$ such that
every neighborhood $U \in \mathcal B_x$
has a local product structure, that is, there exists a stratum-preserving
homeomorphism
\begin{equation}
      U \cong \Rrr^{\dim(X)} \times \cone(\link_W(x))
\label{equation_local-product}
\end{equation}
that is smooth on each stratum which takes the basepoint
$\{x\}$ to $\{0\} \times \{\mbox{cone point}\}$.
\label{lemma_local_product}
\end{lemma}

It follows that a Whitney stratification of a closed set $W \subset M$ is
{\em locally trivial}:
if $x,y \in X$ and if~$\mathcal B_x, \mathcal B_y$ denote the corresponding bases of
neighborhoods of $x$ and $y$ (respectively) then for any $U_x \in \mathcal B_x$
and any $U_y \in \mathcal B_y$  there is a stratum-preserving
homeomorphism
\begin{equation}\label{equation_local_triviality}
h:U_x \longto U_y\end{equation}
that is smooth on each stratum such that $h(x) = y$.

\begin{remark}
{\rm
Complex algebraic, complex analytic, real algebraic, real analytic,
semi-algebraic, semi-analytic, and sub-analytic sets all admit Whitney
stratifications~\cite{Denkowska,Hardt,Hironaka1,Hironaka2,Lojasiewicz,
Mather2,Thom,Verdier}.
The existence of Whitney stratifications, together with the local
triviality theorems, constitute one of the great triumphs of
stratification theory because they provide a deep understanding
of the local structure of the singularities of algebraic and analytic
sets.  In particular they say that these sets do not exhibit fractal
or Cantor set-like behavior.  
}
\end{remark}

\begin{remark}
{\rm
An example of an algebraic set $W$ with a
decomposition into smooth manifolds that is {\em not} locally trivial
is provided by {\em Whitney's cusp}.
See~\cite[Example~2.6]{Mather}
and~\cite{Whitney}.
It is the variety
$x^{3} + y^{2} = x^{2} z^{2}$.  It can be decomposed into two smooth
manifolds as follows.  
The ``small stratum" is the line $x=y=0$, and
all the rest is the ``large stratum".  This decomposition is not
locally trivial at the origin, although it satisfies Whitney's
condition ($A$) everywhere.  Thus, condition ($A$) does not suffice
to guarantee local triviality of the stratification.  Whitney
guessed (correctly) that requiring condition ($B$) would restore
local triviality.  Indeed the above decomposition fails condition
($B$) at the origin.  To restore condition ($B$) it is necessary to
refine the stratification by declaring the origin to be a third
stratum.  The resulting decomposition is locally trivial.
}
\end{remark}

We next state the key definition for developing
face incidence enumeration for 
Whitney stratified spaces.

\begin{definition}
  Let $W$ be a Whitney stratified closed subset of a smooth
  manifold $M$.  Define the face poset ${\mathcal F} = \mathcal F(W)$
  of $W$ to  be the quasi-graded poset consisting of
  the poset of strata ${\mathcal P}$ adjoined with
  a minimal element $\hz$.
  The rank function is given by
  $$      \rho(X) = \begin{cases}
      \dim(X) + 1  & \text{ if } X > \hz, \\
      0           & \text{ if } X = \hz,
    \end{cases}
  $$
  and the weighted zeta function is
  $$
  \zetabar(X,Y)
  =
  \begin{cases}
      \chi(\link_{Y}(X))   & \text{ if } X > \hz, \\
      \chi(Y)              & \text{ if } X = \hz.
    \end{cases}
  $$
  \label{definition_main_definition}
\end{definition}

\begin{theorem}
  Let $W$ be Whitney stratified closed subset of
  a smooth manifold $M$.
  Then the face poset of $W$
  is an Eulerian quasi-graded poset.
  \label{theorem_main_theorem}
\end{theorem}

The proof of Theorem~\ref{theorem_main_theorem}
occupies Sections~\ref{section_control_data}
through~\ref{section_proof}.

We now give a few examples of Whitney stratifications
beginning with the classical polygon.
\begin{example}
  {\rm
    Consider a two dimensional cell $c$ with its boundary
    subdivided into $n$ vertices $v_{1}, \ldots, v_{n}$
    and $n$ edges $e_{1}, \ldots, e_{n}$.
    There are three ways to view this as a Whitney stratification.
    \vspace*{-4mm}
    \begin{itemize}
    \item[(1)]
      Declare each of the $2n+1$ cells to be individual strata.
      This is the classical view of an $n$-gon. Here the
      weighted zeta function is the classical zeta function,
      that is, always equal to $1$
      (assuming $n \geq 2$).

    \item[(2)]
      Declare each of the $n$ edges to be one stratum
      $e = \cup_{i=1}^{n} e_{i}$, that is,
      we have the $n+2$ strata $v_{1}, \ldots, v_{n}, e, c$.
      Here the non-one values of the weighted zeta function
      are given by $\zetabar(\hz,e) = n$ and $\zetabar(v_{i},e) = 2$.

    \item[(3)]
      Lastly, we can have the three strata
      $v = \cup_{i=1}^{n} v_{i}$,
      $e = \cup_{i=1}^{n} e_{i}$ and $c$.
      Now non-one values of the weighted zeta function
      are given by
      $\zetabar(\hz,v) = \zetabar(\hz,e) = n$
      and $\zetabar(v,e) = 2$.
    \end{itemize}
    \vspace*{-4mm}
    In contrast, we cannot have $v, e_{1}, \ldots, e_{n}, c$
    as a stratification, since the link of a point $p$
    in $e_{i}$ depends on the point $p$ in $v$ chosen.
  }
\label{example_three_stratifications}
\end{example}
By Lemma~\ref{lemma_x_y_w}
the $\cd$-index of each of the three Whitney stratifications
in Example~\ref{example_three_stratifications}
are the same, that is, $\cv^{2} + (n-2) \cdot \dv$.
Hence we have the immediate corollary.
\begin{corollary}
  The $\cd$-index of an $n$-gon is given by
  $\cv^{2} + (n-2) \cdot \dv$ for $n \geq 1$.
\end{corollary}

The last stratification in the previous example
can be extended to any simple polytope.
\begin{example}
{\rm
Let $P$ be an $n$-dimensional simple polytope,
that is, every interval $[x,y]$ in the face lattice,
where $\hz < x \leq y$, is 
isomorphic to a Boolean algebra.
We obtain a different stratification
of the ball by joining all the facets together to one strata.  By
Lemma~\ref{lemma_x_y_w} we note that the $\cd$-index does not change,
since the information is carried in the weighted zeta function.  We
continue by joining all the subfacets together to one strata. Again
the $\cd$-index remains unchanged.  In the end we obtain a stratification
where the union of all the $i$-dimensional faces forms the $i$th
strata. The face poset of this stratification is the $(n+2)$-element
chain $C = \{\hz = x_{0} < x_{1} < \cdots < x_{n+1} = \hz\}$, with
the rank function $\rho(x_{i}) = i$ and weighted zeta function
$\zetabar(\hz,x_{i}) = f_{i-1}(P)$ and
$\zetabar(x_{i},x_{j}) = \binom{n+1-i}{n+1-j}$.
Again, by Lemma~\ref{lemma_x_y_w}
we have $\Psi(C,\rho,\zetabar) = \Psi(P)$.

A similar stratification can be obtained for any
regular polytope.
}
\label{example_simple_polytope}
\end{example}

\begin{example}
  {\rm
    Consider the stratification of an $n$-dimensional
    manifold with boundary, 
    denoted $(M, \partial M)$, into its boundary
    $\partial M$ and its interior $M^{\circ}$.
    The face poset is
    $\{\hz < \partial M  < M^{\circ}\}$ with the elements having
    ranks $0$, $n$ and $n+1$, respectively.
    The weighted zeta function is given by
    $\zetabar(\hz,\partial M) = \chi(\partial M)$,
    $\zetabar(\hz,M^{\circ}) = \chi(M)$
    and
    $\zetabar(\partial M, M^{\circ}) = 1$.
    If $n$ is even then $\partial M$ is an odd-dimensional
    manifold without boundary and hence its Euler characteristic
    is $0$.
    In this case the $\ab$-index is
    $   \Psi(M)
    =
    \chi(M) \cdot (\av-\bv)^{n}$.
    If~$n$ is odd then
    we have the relation $\chi(\partial M) = 2 \cdot \chi(M)$
    and hence the $\ab$-index is given by
    $   \Psi(M)
    =
    \chi(M) \cdot (\av-\bv)^{n}
    +
    2 \cdot \chi(M) \cdot (\av-\bv)^{n-1} \cdot \bv$.
    Passing to the $\cd$-index we conclude
    $$
    \Psi(M)
    =
      \begin{cases}
        \chi(M) \cdot (\cv^{2}-2\dv)^{n/2}
        & \text{  if $n$ is even,} \\
        \chi(M) \cdot (\cv^{2}-2\dv)^{(n-1)/2} \cdot \cv
        & \text{  if $n$ is odd.} \\
      \end{cases}
    $$
  }
  \label{example_manifold}
\end{example}
The next example is a higher dimensional analogue
of the one-gon in Example~\ref{example_one-gon}.
\begin{example}
  {\rm
    Consider the subdivision $\Omega_{n}$ of
    the $n$-dimensional ball $\Bbbb^{n}$ consisting
    of a point~$p$, an $(n-1)$-dimensional cell $c$
    and the interior $b$ of the ball.
    If $n \geq 2$,
    the face poset is
    $\{\hz < p < c  < b\}$ with the elements having
    ranks $0$, $1$, $n$ and $n+1$, respectively.
    In the case $n=1$, the two elements
    $p$ and $c$ are incomparable.
    The weighted zeta function is given by
    $\zetabar(\hz,p) = \zetabar(\hz,c) = \zetabar(\hz,b) = 1$,
    $\zetabar(p,c) = 1 + (-1)^{n}$,
    and
    $\zetabar(p,b) = \zetabar(c,b) = 1$.
    Thus the $\ab$-index is
    \begin{equation}
      \Psi(\Omega_{n})
      =
      (\av-\bv)^{n}
      +
      \bv \cdot (\av-\bv)^{n-1}
      +
      (\av-\bv)^{n-1} \cdot \bv
      +
      (1 + (-1)^{n}) \cdot \bv \cdot (\av-\bv)^{n-2} \cdot \bv .
      \label{equation_Omega}
    \end{equation}
    When $n$ is even the expression \eqref{equation_Omega} simplifies to
    \begin{align}
      \Psi(\Omega_{n})
      & = 
      \av \cdot (\av-\bv)^{n-2} \cdot \av
      +
      \bv \cdot (\av-\bv)^{n-2} \cdot \bv
      \nonumber \\
      & = 
      \frac{1}{2}
      \cdot
      \left[
        (\av-\bv) \cdot (\av-\bv)^{n-2} \cdot (\av-\bv)
        +
        (\av+\bv) \cdot (\av-\bv)^{n-2} \cdot (\av+\bv)
      \right]
      \nonumber \\
      & = 
      \frac{1}{2}
      \cdot
      \left[
        (\cv^2-2\dv)^{n/2}
        +
        \cv \cdot (\cv^{2}-2\dv)^{(n-2)/2} \cdot \cv
      \right] .
    \end{align}
    When $n$ is odd the expression \eqref{equation_Omega} simplifies to
    \begin{align}
      \Psi(\Omega_{n})
      & = 
      \av \cdot (\av-\bv)^{n-2} \cdot \av
      -
      \bv \cdot (\av-\bv)^{n-2} \cdot \bv
      \nonumber \\
      & = 
      \frac{1}{2}
      \cdot
      \left[
        (\av+\bv) \cdot (\av-\bv)^{n-2} \cdot (\av-\bv)
        +
        (\av-\bv) \cdot (\av-\bv)^{n-2} \cdot (\av+\bv)
      \right]
      \nonumber \\
      & = 
      \frac{1}{2}
      \cdot
      \left[
        \cv \cdot (\cv^{2}-2\dv)^{(n-1)/2}
        +
        (\cv^{2}-2\dv)^{(n-1)/2} \cdot \cv
      \right] .
    \end{align}
    As a remark, these $\cd$-polynomials
    played an important role in proving that the
    $\cd$-index of a polytope is coefficient-wise
    minimized on the simplex, namely,
    $\Psi(\Omega_{n}) = (-1)^{n-1} \cdot \alpha_{n}$,
    where $\alpha_{n}$ are defined in~\cite{Billera_Ehrenborg}.
    See Theorem~\ref{theorem_intersections}
    for a generalization of one of the 
    main identities in~\cite{Billera_Ehrenborg}.
    See~\cite{Ehrenborg_Johnston_Rajagopalan_Readdy}
    for related polynomials.
  }
  \label{example_point_sphere}
\end{example}

The weighted zeta function
of a stratification can take
negative values, as our last example
illustrates.
\begin{example}
{\rm
Let $M$ be a solid $3$-dimensional torus
such that its boundary is the $2$-dimensional torus $\Ttt^{2}$.
Consider the stratification of $\Ttt^{2}$ into $n$ points
$V = \{v_{1}, v_{2}, \ldots, v_{n}\}$ and
the $n$-punctured torus $\Ttt^{2} - V$.
Its weighted zeta function is negative on
the interval~$[\hz, \Ttt^{2} - V]$,
that is,
$\zetabar(\hz, \Ttt^{2} - V)
= \chi(\Ttt^{2} - V)  = -n$.
The $\cd$-index of this stratification of the solid torus is
$n \cdot \mdc - n \cdot \mcd$.
}
\label{example_n_punctured_torus}
\end{example}

\section{Properties of the Euler characteristic}
\label{section_properties_of_the_Euler_characteristic}

In this section we state two results
involving the  Euler characteristic 
of the boundary of a manifold
and an inclusion-exclusion expression
for the Euler characteristic.
Both will support results leading to the
proof of Theorem~\ref{theorem_main_theorem}
in Section~\ref{section_proof}.

\begin{lemma}
Let $(A,\partial A)$ be a connected $n$-dimensional manifold
(possibly non-orientable)
with (possibly empty) boundary.
Then the Euler characteristic of the boundary is given by
\begin{equation}
  \chi(\partial A) = \left(1 - (-1)^{\dim(A)}\right) \cdot \chi(A).
\label{equation_chi_boundary}
\end{equation}
\label{lemma_Euler_boundary}
\end{lemma}
\begin{proof}(Sketch)
  This follows from Poincar\'e duality
  $H^{i}(A,\partial A;\Zzz_{2}) \cong H_{n-i}(A;\Zzz_{2})$
  and the long exact homology sequence for the
  pair $(A,\partial A)$, together with the fact
  that the Euler characteristic does not depend on
  the coefficient field of the homology groups.
\end{proof}
\begin{lemma}
  Let $X = \bigcup_{i=1}^{r} U_{i}$ be an open cover
  of a topological space $X$.
  For any $R \subseteq \{1,2, \ldots, r\}$ let
  $U_{R} = \bigcap_{i \in R} U_{i}$.
  Then we have the following inclusion-exclusion formula for the
  Euler characteristic:
  \begin{equation}
    \chi(X) = \sum_{\phi \ne R \subseteq [r]} (-1)^{|R|+1} \cdot \chi(U_{R}).
    \label{equation_MV}
  \end{equation}
\label{lemma_MV}
\end{lemma}
\begin{proof} (Sketch)
  This follows from induction and the Mayer--Vietoris theorem
  or equivalently, from the Mayer--Vietoris spectral sequence.
\end{proof}

\section{Tube systems and control data}
\label{section_control_data}

After Whitney introduced his conditions (A) and (B), Ren\'e Thom,
in a remarkable
paper~\cite{Thom}, described a daring technique for proving that
Whitney stratifications
were locally trivial.  Thom constructed the required
homeomorphism~(\ref{equation_local-product}--\ref{equation_local_triviality})
as the time 1
continuous flow of a certain discontinuous (but ``controlled") vector field.
Although
his outline was sound, there were many serious gaps and difficulties
with the
exposition that were finally resolved in J.\ Mather's 
wonderful notes~\cite{Mather},
which are still the best source for this material.  The first step involves the
construction of ``control data", that is,
a system of compatible tubular neighborhoods
of the strata.  Further references for tube systems include
Section~2.5 in~\cite{du_Plessis_Wall},
Section~2.2 in~\cite{Gibson_Wirthmuller_du_Plessis_Looijenga}
and Part~I, Section~1.5 in~\cite{Goresky_MacPherson}.

Let $A$ be a smooth submanifold of a smooth manifold $M$.  Let $\nu(A)$ denote the
normal bundle of~$A$ in $M$.  Choose a smooth inner product
$\langle \cdot, \cdot \rangle$ on $\nu(A)$ and let
$\nu(A)_{< \epsilon} = \left\{ v \in \nu(A):\ \langle v,v \rangle < \epsilon\right\}$.
This quantity may be thought of as
the set of vectors whose distance squared from $A$ is less
than $\epsilon$. A {\em tubular neighborhood} of
$A$ is such a choice of inner product on the normal bundle $\nu(A)$ together with a smooth
embedding $\phi:\nu(A)_{< \epsilon}\longto M$ of the $\epsilon$ neighborhood of the
zero section 
such that the restriction of $\phi$ to the zero section is the identity
map.  It follows that the image
$T_A = T_A(\epsilon) \subset M$ is an open neighborhood
of~$A$ in $M$.
The projection $\pi:\nu(A) \longto A$ determines
a smooth projection $\pi_A:T_A \longto A$
and the inner product determines a smooth ``tubular function"
$\rho_A: T_A \longto [0,\epsilon)$
such that $(\pi_A,\rho_A):T_A \longto A \times (0,\epsilon)$
is a smooth proper submersion (and
hence is a smooth fiber bundle).

Let $W \subset M$ be a closed subset with a fixed Whitney stratification.  A
{\em system of control data} on $W$ consists of
a choice of tubular neighborhood
$T_A$ for each stratum $A$ of $W$ with the following properties:
\begin{enumerate}
\item
$T_A \cap T_B = \phi$ unless $A<B$ or $B <A$,
\item
if $A<B$ then $\pi_A \pi_B(x) = \pi_A(x)$ for all $x \in T_A \cap T_B$
\item
if $A<B$ then $\rho_A \pi_B(x) = \rho_A(x)$ for all $x \in T_A \cap T_B$
\item
The mapping $(\pi_A,\rho_A): W \cap T_A \longto A \times (0,\epsilon)$ is a
stratified submersion, that is, a locally trivial fiber bundle
(with fiber $\link_W(A)$) whose restriction to each stratum is a submersion.
\end{enumerate}
Every Whitney stratified set admits a system of control data,
and we henceforth assume that control data for $W$ has been chosen.

Let $A$ be a stratum of $W$ and
let $\Sigma(A) = \overline{A}-A$ be the closed union of strata in the closure
of~$A$.
We may assume that $\rho_A$ extends to the  $\overline{T}_A - \Sigma(A)$
and takes the value $\epsilon$ on the boundary~$\partial T_A$ which is the
image of the $\epsilon$-sphere bundle in the normal bundle of $A$.
If $\epsilon$ is chosen
sufficiently small then the manifold $\partial T_A \subset M$
is transverse to all
the strata $B>A$ (and it does not intersect any strata $C < A$)
so it is Whitney stratified
by its intersection with the strata of $W$.

For each stratum $A$ define $A^0 = A - \bigcup_{B<A} \overline{T}_B$.
This is an open subset of the stratum $A$. The
family of lines can be used to describe a homeomorphism $A^0 \longto A$.
Moreover, $A^0$ is naturally the interior of a
manifold with boundary (corners),
$A^1 := A - \bigcup_{B<A} T_B\subset A$
whose boundary is a union
$$
\partial A^1 = \bigcup_{B<A} \left(\partial T_B - \bigcup_{C \ne B} T_C\right)
$$
of pieces of the sphere bundles of strata $B < A$.  Let $\overline{X}$ be the closure of a
single stratum of $W$.  Then the system of control data
$(T_A, \pi_A, \rho_A, \Phi^A)$ when intersected with $\overline{X}$ gives a system of control
data for $\overline{X}$.
These notions are illustrated in
Figures~\ref{figure_one}
through~\ref{figure_three}.

If $A$ is a stratum of $W$ and if $x \in A$ then $\pi_A^{-1}(x) = N_x$ is
a ``normal slice'' to $A$ at $x$ as described in 
Section~\ref{sec-stratified-sets}.
Thus the link of the stratum $A,$ at the point $x \in A$ is (homeomorphic to) the set
$$
\link(A)=\pi_A^{-1}(x) \cap \partial T_A \cap W.
$$
For any stratum $B>A$
the {\em link of the stratum $A$ in $B$} is the
(not necessarily compact, not necessarily connected) manifold
$$
\link_B(A)=\pi_A^{-1}(x) \cap \partial T_A \cap B.
$$

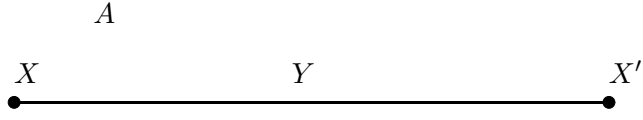
\begin{figure}[H]
\setlength{\unitlength}{1.5pt}
\centerline{
\begin{picture}(200,40)
\put(40,00){\circle*{3}}
\qbezier(40,0)(110,0)(190,0)
\put(190,0){\circle*{3}}
\thicklines
\put(40,5){$X$}\put(60,20){$A$}
\put(190,5){$X'$}
\put(110,5){$Y$}
\end{picture}
}
\caption{Strata: $X<Y<A$}
\label{figure_one}
\end{figure}

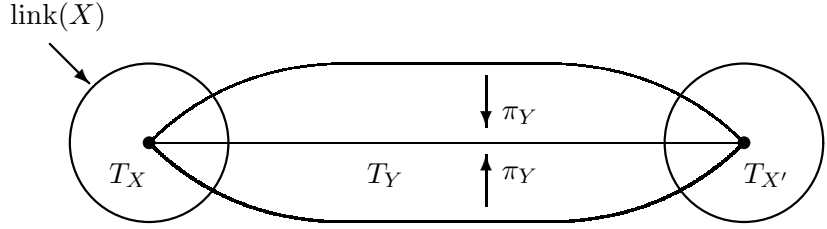
\begin{figure}[H]
\setlength{\unitlength}{1.5pt}
\centerline{
\begin{picture}(200,60)(0,20)
\put(40,40){\circle*{3}}
\linethickness{.5pt}
\qbezier(40,40)(110,40)(190,40)
\linethickness{.5pt}
\qbezier(40,40)(60,60)(90,60)
\qbezier(90,60)(120,60)(140,60)
\qbezier(140,60)(170,60)(190,40)
\put(190,40){\circle*{3}}
\thicklines
\put(40,40){\circle{40}}
\put(190,40){\circle{40}}
\qbezier(40,40)(60,20)(90,20)
\qbezier(90,20)(120,20)(140,20)
\qbezier(140,20)(170,20)(190,40)
\put(5,70){$\link(X)$}
\put(15,65){\vector(1,-1){10}}
\put(125,24){\vector(0,1){13}}\put(129,46){$\pi_Y$}
\put(125,57){\vector(0,-1){13}}\put(129,31){$\pi_Y$}
\put(30,30){$T_X$}
\put(95,30){$T_{Y}$}
\put(190,30){$T_{X'}$}
\end{picture}
}
\caption{Tubular neighborhoods}
\label{figure_two}
\end{figure}

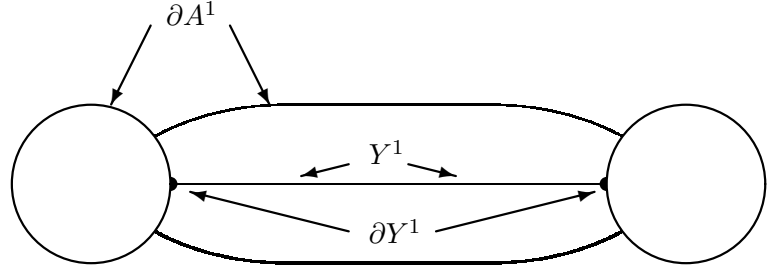
\begin{figure}[H]
\setlength{\unitlength}{1.5pt}
\centerline{
\begin{picture}(200,80)(0,20)
\put(40,40){\circle*{3}}
\linethickness{.5pt}
\qbezier(40,40)(110,40)(190,40)
\linethickness{.5pt}
\qbezier(40,40)(60,60)(90,60)
\qbezier(90,60)(120,60)(140,60)
\qbezier(140,60)(170,60)(190,40)
\put(190,40){\circle*{3}}
\thicklines
\qbezier(40,40)(60,20)(90,20)
\qbezier(90,20)(120,20)(140,20)
\qbezier(140,20)(170,20)(190,40)
\put(60,40){\circle*{3}}\put(170,40){\circle*{3}}
\put(110,45){$Y^1$}
\put(120,45){\vector(4,-1){12}}
\put(105,45){\vector(-4,-1){12}}
\put(110,25){$\partial Y^1$}
\put(105,28){\vector(-4,1){40}}
\put(127,28){\vector(4,1){40}}
\put(59,80){$\partial A^1$}
\put(55,80){\vector(-1,-2){10}}
\put(75,80){\vector(1,-2){10}}
\filltype{white}
\put(40,40){\circle*{40}}
\put(190,40){\circle*{40}}
\end{picture}
}
\caption{Manifolds with boundary}
\label{figure_three}
\end{figure}

\begin{proposition}
 Let $B_1<B_2<\cdots<B_{k}$ be a chain of strata and
let $\pi_j:T_{B_j} \longto B_j$ be the projection. Define
\begin{equation}\label{equation_big_intersection}
 \partial T(B_1,B_2,\ldots,B_{k}):=
\partial T_{B_1} \cap \partial T_{B_2} \cap \cdots \cap \partial T_{B_{k}}.
\end{equation}
If $\epsilon >0$ is sufficiently small then for any stratum $A$ of $W$ the
Euler characteristic of the intersection
$\partial T(B_1,B_2,\ldots,B_{k}) \cap A$ factors:
\begin{equation}
  \chi\left(\partial T(B_1,B_2,\ldots,B_{k})\cap A\right)
  =
  \chi(B_1)
    \cdot
  \chi(\link_{B_2}(B_1))
    \cdot
  \chi(\link_{B_3}(B_2))
    \cdots
  \chi(\link_{A}(B_{k})).
\label{equation_product}
\end{equation}
\label{lemma_product}
\label{proposition_Euler_characteristic_factors}
\end{proposition}

\begin{proof} To simplify the notation, we write $\pi_1,\rho_1$ rather than
$\pi_{B_1}, \rho_{B_1}$, etc.  We will prove by induction on $t$ that if
$\epsilon$ is sufficiently small then
\begin{enumerate}
\item[(a)] the collection of submanifolds
$\partial T_{B_i}\subset M$ is mutually transverse,
\item[(b)] $\partial T(B_1,B_2,\ldots,B_{k})$ is transverse to every stratum $A$ of $W$, and
\item[(c)] the intersection $\partial T(B_1,B_2,\ldots,B_{k})\cap W$
may be described as an
iterated stratified fibration:
\begin{diagram}[size=1.3em]
  &&\link_W(B_{k})\\
  &\ldTo & \\
  \partial T(B_1,B_2,\ldots,B_{k})\cap W&&\link(B_{k-1})\cap B_{k}\\ \dTo_{\pi_t}& \ldTo&\\
  \partial T(B_1,B_2,\ldots,B_{k-1})\cap B_{k}&&\qquad \link(B_{k-2})\cap B_{k-1}\\
  \dTo_{\pi_{t-1}}& \ldTo& \\
  \partial T(B_1,B_2,\ldots,B_{k-2}) \cap B_{k-1}&& \\ \dTo&& \\
  \cdots&& \link(B_1)\cap B_2\\
  \dTo& \ldTo& \\
  \partial T(B_1) \cap B_2 &&\\
  \dTo_{\pi_1}&&\\ B_1.&&\end{diagram}\end{enumerate}
Consequently, for any stratum $A>B_{k}$  the intersection $\partial T(B_1,B_2,\ldots,B_{k})\cap A$
also fibers in this way, from which equation ({\ref{equation_product}}) follows immediately.

For $k=1$,
part (a) is vacuous and parts (b) and (c) are simply a restatement of
property (4) in the definition of control data.

For the inductive step, suppose we have proven (a), (b), (c) for a chain of strata
$B_2 < B_3 < \cdots < B_{k}$.  In particular,
\begin{equation}
\partial T(B_2,B_3,\ldots,B_{k}) \cap W \longto \cdots \longto B_2
\label{equation_partial_tower}
\end{equation}
is an iterated fibration.  Consider the effect of adding a stratum $B_1$ to
the bottom of this chain.   By property (4), the mapping
$$
   (\pi_1,\rho_1):T(B_1)\cap W \longto B_1 \times (0,\epsilon)
$$
is a stratified fibration:  it is a submersion on each stratum.
Composing this with the tower~\eqref{equation_partial_tower} gives an
iterated stratified fibration
$$
  \partial T(B_2,B_3,\ldots,B_{k})\cap T(B_1)\cap W \longto \cdots \longto
B_2 \cap T(B_1) \longto B_1 \times (0,\epsilon).
$$
The projection to $(0,\epsilon)$ is a submersion on each stratum, which is
equivalent to the statement that for all $s \in (0,\epsilon)$ the submanifold
$\rho_1^{-1}(s)\subset M$ is transverse to every stratum in this tower.
Shrinking~$\epsilon$ if necessary, the submanifold $\rho_1^{-1}(\epsilon) =
\partial T(B_1)$ is therefore transverse to every stratum that appears in this
tower.  Thus, restricting to $\rho_1^{-1}(\epsilon)$  gives
the desired iterated stratified fibration
\[
\partial T(B_2,\ldots,B_{k}) \cap \partial T(B_1) \longto \cdots \longto
B_2 \cap \partial T(B_1) \longto B_1. \qedhere
\]
\end{proof}

\section{The main formula}
\label{section_proof}

In this section
we will  give the proof of our main result from
Section~\ref{sec-stratified-sets}, that the face poset
of a Whitney stratified closed subset of
a smooth manifold is indeed an Eulerian quasi-graded poset.

Recall that we have homotopy equivalences $A^0 \subset A^1 \subset A$
for a stratum $A$.
\begin{proposition}
  Let $A$ be a stratum of $W$.
  Then the boundary $\partial A^1$ is a topological manifold and
  \begin{equation}
    \chi(\partial A^1)
    =
    \sum_{B<A}
    (-1)^{\dim(B)} \cdot \chi(B) \cdot \chi(\link_{A}(B)).
    \label{equation_main_one}
  \end{equation}
  Consequently, the following identity holds:
  \begin{equation}
    \left(1 + (-1)^{\dim(A)+1}\right) \cdot \chi(A)
    +
    \sum_{B < A}
    (-1)^{\dim(B)+1} \cdot \chi(B) \cdot \chi(\link_{A}(B))
    =
    0.
    \label{equation_main_two}
  \end{equation}
  \label{proposition_main}
\end{proposition}
\begin{proof}
Let us cover the boundary $\partial A^1$ with $\epsilon$-neighborhoods of the
closed sets $S_B = \partial T_B - \bigcup_{C \ne B} T_C$ as~$B$ ranges over
the strata $B<A$.  By equation~\eqref{equation_MV}
of Lemma~\ref{lemma_MV},
the Euler characteristic
$\chi(\partial A^1)$ is an
alternating sum of the Euler characteristics of the multi-intersections of
these sets $S_B$.  As observed 
in Section~\ref{section_control_data}
such an intersection is empty unless
the corresponding strata form a chain.  Let ${\mathcal P}$ denote the partially
ordered set of strata~$B < A$.  If $\bR \subseteq {\mathcal P}$ let
$\partial T(\bR) = \bigcap_{B \in \bR} S_B$.
By Lemma~\ref{lemma_MV} we have
  \[
  \chi(\partial A^1)
  =
  \sum_{\phi \ne \bR \subseteq {\mathcal P}}
  (-1)^{|\bR|+1} \cdot \chi(\partial T(\bR) \cap A).\]
  As observed in Proposition~\ref{proposition_Euler_characteristic_factors},
the term $\chi(\partial T(\bR)\cap A)$ is zero unless
  $\bR = (R_1<R_2<\cdots<R_{k})$ is a chain
  of strata, with $R_{k} <A$, in which case the Euler characteristic
is given by equation~\eqref{equation_product}.
We will need to focus on the top element $R_{|\bR|}$ of each
chain.  For notational convenience let us define
  \[ \chi(\bR)
  =
  \chi(R_1)
  \cdot
  \chi(\link_{R_2}(R_1))
  \cdot
  \chi(\link_{R_3}(R_2))
  \cdots
  \chi(\link_{R_{k}}(R_{k-1})).\]
  Then equation~\eqref{equation_product} gives
  $\chi(\partial T(\bR)\cap A) = \chi(\bR) \cdot \chi(\link_A(R_{k}))$
  and we have
  \begin{equation}
    \chi(\partial A^1)
    =
    \sum_{\bR<A}
    (-1)^{|\bR|+1} \cdot \chi(\bR) \cdot \chi(\link_A(R_{|\bR|})) .
    \label{equation_MV2}
  \end{equation}

  If $|\bR|>1$ then it is possible to remove the top element
  and obtain a smaller chain
  $
    \bT = (T_1<T_2<\cdots<T_{k-1})=(R_1<R_2<\cdots<R_{k-1})
  $
  such that $\bR = (\bT<R_{|\bR|})$.  Let $B = R_{k} = R_{|\bR|}$ denote
  this top element and note that $R_{k-1} = T_{|\bT|}$.
  Then the corresponding term in equation~\eqref{equation_MV2} is therefore
  \[ (-1)^{|\bR|+1}
  \cdot
  \chi(\bT)
  \cdot
  \chi(\link_{B}(T_{|\bT|}))
  \cdot
  \chi(\link_A(B)).\]
  If $|\bR|=1$
  then the chain $\bR$ consists of a single stratum $B \in {\mathcal P}$
  and the
  contribution from this term is just $\chi(B) \cdot \chi(\link_A(B))$.

  With this notation we are able to group together the terms involving
  $\chi(\link_A(B))$ as
  $B \in{\mathcal P}$ varies, that is,
  we group terms according to the last factor
  in equation~\eqref{equation_product} to obtain
  \[ \chi(\partial A^1)
  =
  \sum_{B \in {\mathcal P}}
  \chi(\link_A(B))
  \cdot
  \left(\chi(B)
    +
    \sum_{\bT<B}
    (-1)^{|\bT|} \cdot \chi(\bT) \cdot \chi(\link_B(T_{|\bT|}))
  \right)  . \]
  By equation~\eqref{equation_MV2} the inner sum
is just $-\chi(\partial B^1)$, so the quantity
inside
  the parentheses is
  $ \chi(B) - \chi(\partial B^1) = (-1)^{\dim(B)} \cdot \chi(B)$
  by equation~\eqref{equation_chi_boundary}.  Consequently
  \[ \chi(\partial A^1)
  =
  \sum_{B \in {\mathcal P}}
  \chi(\link_A(B)) \cdot (-1)^{\dim{B}} \cdot \chi(B), \]
  which is equation~\eqref{equation_main_one}.
  Equation~\eqref{equation_main_two} follows from
  equation~\eqref{equation_main_one} and
  Lemma~\ref{lemma_Euler_boundary}.
\end{proof}

We are now ready to give the proof of the main result from
Section~\ref{sec-stratified-sets}.
\begin{proof}[Proof of Theorem~\ref{theorem_main_theorem}]
  Let $X < Z$ be elements of the face poset ${\mathcal F}(P)$,
  where we allow the possibility that $X = \hz$.
  Let $W' := \link_W(X) \cap \overline{Z}$
  be the intersection of the link of $X$ with the closure
  of~$Z$. This is again a Whitney stratified subset of $M$ and its poset
  ${\mathcal P}'$ of strata is equal to the interval $[X,Z] \subseteq
  {\mathcal P}$ because the strata of~$W'$ are all of the form
  $Y' = Y \cap \link_W(X) $ where $X \le Y \le Z$.
  The minimal element $\hz'$ of
  ${\mathcal P}'$ corresponds to the empty stratum,
  that is, $\hz' = X \cap \link_W(X) = \phi$.
  Moreover, we claim that the weighted zeta function
  for $W'$, denoted by $\zetabar'$,
  coincides with that of~$W$ restricted to the
  interval $[X,Z],$ that is,
  $$
    \zetabar(X,Y) = \zetabar'(X',Y').
  $$
  In fact, by equation~\eqref{equation_link_intersection},
  taking $P = \link_M(X),$ there is a stratum-preserving homeomorphism
  \begin{align*}
    \link_W(Y) &\cong \link_{W\cap P}(Y\cap P)\\
  \intertext{and hence, intersecting with $\overline{Z}$ gives a homeomorphism}
    \link_W(Y) \cap \overline{Z} &\cong \link(Y') \cap \overline{Z}'
  \end{align*}
  where $Y' = Y\cap P =Y\cap\link(X)$
  and $Z' = Z \cap P=Z\cap\link(X)$ are the corresponding strata of~$W'$.

The rank function $\rho'$ of $\mathcal P'$ is
\[\rho(Y') = \dim(Y')+1 = \dim(Y)+1-\dim(X)-1 = \rho(X,Y) .\]
Therefore we may apply
Equation~\eqref{equation_main_two}
to the stratified space $W'$ and thereby obtain
\[  \sum_{X \le Y \le Z}
  (-1)^{\rho(X,Y)} \cdot \zetabar(X,Y) \cdot \zetabar(Y,Z)
 =
  \delta_{X,Z}.   \qedhere  \]
\end{proof}

\section{The semisuspension}
\label{section_the_semisuspension}

Let $\Gamma$ be a polytopal complex,
that is, a regular cell complex whose cells are polytopes.
Assume the dimension of $\Gamma$ is $k$.
Let $n > k$ be an integer.
We define the {\em $n$th semisuspension}
of $\Gamma$, 
denoted
$\Semi(\Gamma,n)$, to be
the family of $CW$~complexes
obtained by embedding~$\Gamma$ in the boundary of an $n$-dimensional
ball $\Bbbb^{n}$.
Thus we are adding the two strata
$\partial \Bbbb^{n} - \Gamma$ and the
interior of $\Bbbb^{n}$ to the complex~$\Gamma$
to obtain the stratification.
Note that one really has a family of embeddings.
For example, 
one can embed 
a circle into the boundary
of a $4$-dimensional ball so that the result is any given knot.
Nevertheless, we will show the face poset of
$\Semi(\Gamma,n)$ is well-defined.
Furthermore, in the case
$\Gamma$ is homeomorphic to a $k$-dimensional ball,
the semisuspension $\Semi(\Gamma,n)$ is unique
up to homeomorphism.

The {\em face poset} ${\mathscr F}(\Semi(\Gamma,n))$ 
of the $n$th semisuspension of $\Gamma$
consists of
the face poset of the complex~$\Gamma$ with two extra elements
$*$ and~$\ho$ of rank $n$ and $n+1$, respectively,
where the element $*$ corresponds to $\Sss^{n-1} - \Gamma$
and the maximal element corresponds
to the interior of the ball $\Bbbb^{n}$.
The order relations are $x \leq *$ for $x \in \Gamma$ and
$* \coveredby \ho$.
The weighted zeta function $\zetabar$ is given by
\begin{equation}
       \zetabar(x,y)
    =
               \begin{cases}
 \chi(\Sss^{n-1-\rho(x)} - \link_{\Gamma}(x))
& \text{ if } x < * \text{ and } y = *, \\
                   1
& \text{ otherwise.}
               \end{cases}   
             \label{equation_zetabar_semi_I}
\end{equation}
Observe that when $n=k+1$ the element $*$ would have the same
rank as any of the $k$-dimensional facets~$x$
and hence
$\zetabar(x,*) = 0$.
By Alexander duality we know that for a closed subset $A$ of $\Sss^{m}$
the Euler characteristic of the complement
$\chi(\Sss^{m}-A)$ is given by
$\chi(A)$ if $m$ is odd and $2-\chi(A)$ if $m$ is even.
Using the reduced Euler characteristic
we can write this relation as
$\chi(\Sss^{m}-A)
  =
 1 - (-1)^{m} \cdot \widetilde{\chi}(A)$.
Thus the weighted zeta function can be rewritten as
\begin{equation}
       \zetabar(x,y)
    =
                \begin{cases}
 1 + (-1)^{\rho(x,*)} \cdot \widetilde{\chi}(\link_{\Gamma}(x))
& \text{ if } x < y \text{ and } y = *, \\
                   1
& \text{ otherwise.}
               \end{cases}    
             \label{equation_zetabar_semi_II}
\end{equation}
Also note that for a facet $x$ of $\Gamma$, the link
$\link_{\Gamma}(x)$ is the complex consisting of the
empty set, and hence its reduced Euler characteristic is $-1$.

Note that equation~\eqref{equation_zetabar_semi_II}
does not depend on the particular embedding of the
complex $\Gamma$ into the $n$-dimensional sphere.
Summarizing this discussion and using
Theorem~\ref{theorem_main_theorem},
we have the following result.
\begin{proposition}
Let $\Gamma$ be a $k$-dimensional polytopal complex
and let $n > k$ be an integer.
Then the face poset of the $n$th semisuspension
${\mathscr F}(\Semi(\Gamma,n))$ 
having
the weighted zeta function $\zetabar$ in~\eqref{equation_zetabar_semi_I}
does not
depend on the embedding of $\Gamma$ into the
boundary of $\Bbbb^{n}$.
Furthermore, 
the face poset
${\mathcal F}(\Semi(\Gamma,n))$ is an Eulerian quasi-graded poset.
\end{proposition}

We now can relate the $\cd$-index of the $(n+1)$st semisuspension
to that of the  $n$th semisuspension.

\begin{proposition}
Let $\Gamma$ be a polytopal complex of dimension
less than $n$. Then the following identity holds:
$$   \Psi(\Semi(\Gamma,n+1))
  =
     \Psi(\Semi(\Gamma,n)) \cdot \cv
  -
     \Psi([\hz,*]) \cdot \dv  ,  $$
where the interval $[\hz,*]$ occurs in 
the face poset of the semisuspension $\Semi(\Gamma,n))$.
\label{proposition_Wednesday}
\end{proposition}
\begin{proof}
We begin by expanding the $\cd$-index of
the semisuspension $\Semi(\Gamma,n+1)$
using the chain definition of the $\ab$-index.
Note there are four types of chains,
depending on
whether
the $*$-element is in the chain or not,
and
whether
the chain contains a non-empty element of the complex $\Gamma$
or not.
In order to distinguish between the two face posets
${\mathcal F}(\Semi(\Gamma,n))$
and
${\mathcal F}(\Semi(\Gamma,n+1))$,
we mark the weighted zeta function, the rank function
and the element $*$
of the second poset with primes.
We have
\begin{align*}
\Psi(\Semi(\Gamma,n+1))
  & = 
(\av-\bv)^{n+1}
    +
\sum_{x \in \Gamma - \{\hz\}}
         \Psi([\hz,x]) \cdot \bv \cdot (\av-\bv)^{\rho^{\prime}(x,\ho)-1} \\
  &   
    +
\zetabar^{\prime}(\hz,*^{\prime}) \cdot (\av-\bv)^{n} \cdot \bv
    +
\sum_{x \in \Gamma - \{\hz\}}
         \Psi([\hz,x]) \cdot \bv \cdot \zetabar^{\prime}(x,*^{\prime}) \cdot
         (\av-\bv)^{\rho^{\prime}(x,*^{\prime})-1} \cdot \bv .\\
\end{align*}
Note that
$\zetabar^{\prime}(x,*^{\prime})
  =
1 + (-1)^{\rho^{\prime}(x,*^{\prime})} \cdot \widetilde{\chi}(\link_{\Gamma}(x))
  =
2 - \left(1 + (-1)^{\rho(x,*)} \cdot \widetilde{\chi}(\link_{\Gamma}(x))\right)
  =
2 - \zetabar(x,*)$.
Hence the above expression translates to
\begin{align*}
\Psi(\Semi(\Gamma,n+1))
  & = 
(\av-\bv)^{n+1}
    +
\sum_{x \in \Gamma - \{\hz\}}
         \Psi([\hz,x]) \cdot \bv \cdot
         (\av-\bv)^{\rho(x,\ho)-1} \cdot (\av-\bv) \\
  &   
  + (2 - \zetabar(\hz,*)) \cdot (\av-\bv)^{n} \cdot \bv\\
  &   
    +
\sum_{x \in \Gamma - \{\hz\}}
         \Psi([\hz,x]) \cdot \bv \cdot (2-\zetabar(x,*)) \cdot
         (\av-\bv)^{\rho(x,*)-1} \cdot (\av-\bv) \cdot \bv \\
  & = 
(\av-\bv)^{n} \cdot \cv
    +
\sum_{x \in \Gamma - \{\hz\}}
         \Psi([\hz,x]) \cdot \bv \cdot
         (\av-\bv)^{\rho(x,\ho)-1} \cdot \cv \\
  &   
    -
\zetabar(\hz,*) \cdot (\av-\bv)^{n-1} \cdot (\av-\bv) \cdot \bv \\
  &   
    -
\sum_{x \in \Gamma - \{\hz\}}
         \Psi([\hz,x]) \cdot \bv \cdot \zetabar(x,*) \cdot
         (\av-\bv)^{\rho(x,*)-1} \cdot (\av-\bv) \cdot \bv  \\
  & = 
     \Psi(\Semi(\Gamma,n)) \cdot \cv
    -
\zetabar(\hz,*) \cdot (\av-\bv)^{n-1} \cdot \dv \\
  &   
    -
\sum_{x \in \Gamma - \{\hz\}}
         \Psi([\hz,x]) \cdot \bv \cdot \zetabar(x,*) \cdot
         (\av-\bv)^{\rho(x,*)-1} \cdot \dv  \\
  & = 
     \Psi(\Semi(\Gamma,n)) \cdot \cv
    -
     \Psi([\hz,*]) \cdot \dv  ,
\end{align*}
where we used 
$-(\av-\bv) \cdot \bv = \bv \cdot \cv - \dv$
in the third step.
\end{proof}

\section{Inclusion-exclusion for the semisuspension and its $\cd$-index}
\label{section_inclusion_exclusion_semisuspension}

This section begins with an inclusion-exclusion relation for
the $n$th semisuspension of polytopal complexes.

\begin{theorem}
Let $\Gamma$
and $\Delta$
be two polytopal complexes
such that
their union $\Gamma \cup \Delta$ is a polytopal complex of dimension
less than $n$.
Then the following inclusion-exclusion relation holds:
\begin{equation}
     \Psi(\Semi(\Gamma,n)) + \Psi(\Semi(\Delta,n))
   =
     \Psi(\Semi(\Gamma \cap \Delta,n)) + \Psi(\Semi(\Gamma \cup \Delta,n)) .
\label{equation_semisuspension_inclusion_exclusion}
\end{equation}
\label{theorem_inclusion-exclusion}
\end{theorem}
\begin{proof}
Consider a chain $c = \{\hz < x_{1} < \cdots < x_{i} = x < \ho\}$ 
in the face poset of 
$\Semi(\Gamma \cup \Delta,n)$
whose largest element $x$ belongs to 
$\Gamma \cup \Delta$.
If the element $x$ is in the intersection $\Gamma \cap \Delta$ 
then this chain
is enumerated twice on both sides of the identity.
If the element $x$ lies in $\Gamma$ but not in 
$\Gamma \cap \Delta$ then this chain is enumerated once on both sides.
Symmetrically,
if the element $x$ lies in $\Delta$ but not in 
the intersection then the chain is again enumerated once on both sides.

It remains to consider chains 
$c = \{\hz < x_{1} < \cdots < x_{i} =  x < * < \ho\}$ 
that contain the element $*$.
Again let $x$ be the largest element in the chain contained
in the union $\Gamma \cup \Delta$.
If $x$ does not belong to the intersection $\Gamma \cap \Delta$
then the chain $c$ is enumerated in one term 
from each side of~\eqref{equation_semisuspension_inclusion_exclusion}.
The case that remains
is when $x$ belongs to the intersection $\Gamma \cap \Delta$.
Note that the link of $x$ in
$\Gamma$, $\Delta$, $\Gamma \cap \Delta$
and
$\Gamma \cup \Delta$ are all polytopal complexes
and
the reduced Euler characteristic of polytopal complexes
behaves as a valuation, that is,
$\widetilde{\chi}(\link_{\Gamma}(x))
  +
\widetilde{\chi}(\link_{\Delta}(x))
= 
\widetilde{\chi}(\link_{\Gamma \cap \Delta}(x))
  +
\widetilde{\chi}(\link_{\Gamma \cup \Delta}(x))$.
Hence we have
\begin{align*}
    \zetabar_{\Gamma}(x,*) + \zetabar_{\Delta}(x,*)
  & = 
    1 + (-1)^{\rho(x,*)} \cdot \widetilde{\chi}(\link_{\Gamma}(x))
  +
    1 + (-1)^{\rho(x,*)} \cdot \widetilde{\chi}(\link_{\Delta}(x)) \\
  & = 
    1 + (-1)^{\rho(x,*)} \cdot \widetilde{\chi}(\link_{\Gamma \cap \Delta}(x))
  +
    1 + (-1)^{\rho(x,*)} \cdot \widetilde{\chi}(\link_{\Gamma \cup \Delta}(x)) \\
  & = 
    \zetabar_{\Gamma \cap \Delta}(x,*)
  +
    \zetabar_{\Gamma \cup \Delta}(x,*) ,
\end{align*}
using equation~\eqref{equation_zetabar_semi_II}.
Hence the zeta weight of the chain $c$  satisfies
$\zetabar_{\Gamma}(c) + \zetabar_{\Delta}(c)
=
    \zetabar_{\Gamma \cap \Delta}(c)
  +
    \zetabar_{\Gamma \cup \Delta}(c)$
and thus contributes the same amount to both sides
of~\eqref{equation_semisuspension_inclusion_exclusion}.
\end{proof}

\begin{corollary}
Let $\Gamma_{1}, \ldots, \Gamma_{r}$ be $r$ polytopal complexes
such that their union has dimension less than~$n$. Then the 
following inclusion-exclusion
relation holds:
$$ \Psi\left(\Semi\left(\bigcup_{i=1}^{r} \Gamma_{i} , n\right)\right)
     =
   \sum_{\emptyset \neq I \subseteq \{1, \ldots, r\}}
           (-1)^{|I|-1}
      \cdot
           \Psi\left(\Semi\left(\bigcap_{i \in I} \Gamma_{i} , n \right)\right) .
$$
\label{corollary_inclusion-exclusion}
\end{corollary}

For a polytopal complex $\Gamma$ with facets
$F_{1}, \ldots, F_{m}$, define the
{\em nerve complex} ${\mathcal N}(\Gamma)$
to be the simplicial complex
with vertex set $\{1, \ldots, m\}$ and 
the subset $I$ is a face if
the intersection $\bigcap_{i \in I} F_{i}$ is non-empty.
We need the following version of the
nerve theorem.
A weaker version
is due to Borsuk~\cite[Theorem~1]{Borsuk}.
For other versions and references, see
Bj\"orner's overview article~\cite[Section~4]{Bjorner_handbook}.
The version we state here follows from
Bj\"orner, Korte and Lov\'asz~\cite[Theorem~4.5]{Bjorner_Korte_Lovasz}.
\begin{theorem}[Nerve theorem for polytopal complexes]
For a polytopal complex $\Gamma$,
the complex~$\Gamma$ and the nerve complex
${\mathcal N}(\Gamma)$ are homotopy equivalent.
\label{theorem_polytopal_and_nerve}
\end{theorem}

We can now generalize Proposition~4.3
in~\cite{Billera_Ehrenborg}.
Recall that $\Omega_{n}$
denotes the stratification of an $n$-dimensional closed ball
into a point, an $(n-1)$-dimensional cell and
an $n$-dimensional cell.
See Example~\ref{example_point_sphere}.
\begin{theorem}
Let $\Gamma$ be a polytopal complex of dimension less than~$n$.
Assume that $\Gamma$ has facets $F_{1}, \ldots, F_{r}$.
Then the $\cd$-index of the semisuspension $\Semi(\Gamma,n)$
is given by
$$ \Psi(\Semi(\Gamma,n))
     =
  -
   \sum_{F}
      \widetilde{\chi}(\link_{\Gamma}(F))
        \cdot
      \Psi(F)
        \cdot
      \Psi(\Omega_{n-\dim(F)}) ,
$$
where the sum is over all possible intersections $F$
of the facets $F_{1}, \ldots, F_{r}$.
\label{theorem_intersections}
\end{theorem}
\begin{proof}
By the inclusion-exclusion
Corollary~\ref{corollary_inclusion-exclusion},
we have that
\begin{align}
   \Psi(\Semi(\Gamma, n))
  & = 
   \sum_{\emptyset \neq I \subseteq \{1, \ldots, r\}}
           (-1)^{|I|-1}
      \cdot
           \Psi\left(\Semi\left(\bigcap_{i \in I} F_{i} , n \right)\right)
\nonumber \\
  & = 
   \sum_{F}
\left(
   \sum_{\onethingatopanother{\emptyset \neq I \subseteq \{1, \ldots, r\}}
                             {\bigcap_{i \in I} F_{i} = F}}
           (-1)^{|I|-1}
\right)
      \cdot
           \Psi(\Semi(F,n)) ,
\label{equation_Semi_intersection}
\end{align}
where the outer sum is over all possible intersections $F$ of
the facets $F_{1}, \ldots, F_{r}$. We express the inner sum
using the nerve complex of the link.
For $F$ a face of $\Gamma$
let $J(F)$ be the non-empty index set 
$J(F) = \{i \in I \: : \: F \subseteq F_{i}\}$.
The inner sum of~\eqref{equation_Semi_intersection} is given by
\begin{align}
\sum_{\onethingatopanother{\emptyset \neq I \subseteq J(F)}
                          {\bigcap_{i \in I} F_{i} = F}}
           (-1)^{|I|-1}
  & = 
\sum_{\emptyset \neq I \subseteq J(F)}
           (-1)^{|I|-1}
   -
\sum_{\onethingatopanother{I \subseteq J(F)}
                          {\bigcap_{i \in I} F_{i} \supsetneqq F}}
           (-1)^{|I|-1}
\nonumber \\
  & = 
1
   -
\chi({\mathcal N}(\link_{\Gamma}(F)))
\nonumber \\
  & = 
   -
\widetilde{\chi}({\mathcal N}(\link_{\Gamma}(F)))
\nonumber \\
  & = 
   -
\widetilde{\chi}(\link_{\Gamma}(F)) ,
\label{equation_nerve}
\end{align}
where in the last step we used that the nerve complex
of a polytopal complex is homotopy equivalent to the original
complex and hence they have the same (reduced) Euler characteristic.
Finally, observe that the face poset of
$\Semi(F,n)$ is given by the Stanley product
\begin{equation}
     {\mathcal F}(\Semi(F,n))
   =
     {\mathcal F}(F)
   *
     {\mathcal F}(\Omega_{n-\dim(F)})   . 
\label{equation_F_Omega}
\end{equation}
By combining equations~\eqref{equation_Semi_intersection},
\eqref{equation_nerve}
and~\eqref{equation_F_Omega}, the result follows.
\end{proof}

Theorem~\ref{theorem_intersections}
generalizes Proposition~4.3 in~\cite{Billera_Ehrenborg}
which considered the case when
$F_{1}, \ldots, F_{r}$ is the initial line shelling segment of
an $n$-dimensional polytope. Their proof is based
on shelling, whereas the proof we give here for
Theorem~\ref{theorem_intersections}
is an application of inclusion-exclusion.

\section{The Eulerian relation for the semisuspension}
\label{section_local}

Let $\Gamma$ be a regular subdivision of an
$n$-dimensional ball $\Bbbb^{n}$
such that the interior of the ball is one of
the faces.
Let $\Lambda$ be a regular subdivision of~$\Gamma$
such that the interior of the ball is yet again a
face of~$\Lambda$.
For a face $F$ of $\Gamma$ we define
$\Lambda|_{F}$ to be the
subdivision of $F$ induced by $\Lambda$.
There are two extremal cases.
When $F$ is the empty set, let
$\Lambda|_{F}$ be the empty subdivision of the empty face.
In this case
the semisuspension
$\Semi(\Lambda|_{F},n)$
is the $(n-1)$-dimensional sphere
and the interior of the $n$-dimensional ball.
The second extremal case
is when $F = \ho$,
and we let
$\Lambda|_{F}$ 
and
$\Semi(\Lambda|_{F},n)$
denote the subdivision~$\Lambda$
of the $n$-dimensional sphere.

\begin{theorem}
Let $\Gamma$ be a regular subdivision of the
$n$-ball $\Bbbb^{n}$
and
let $\Lambda$ be a regular subdivision of~$\Gamma$
such that both subdivisions have the interior of the ball as a face.
Then the alternating sum of $\cd$-indexes of
semisuspensions is equal to zero, that is,
$$   \sum_{F \in \Gamma}
          (-1)^{\rho(F,\ho)}
       \cdot
          \Psi(\Semi(\Lambda|_{F},n))   
  =
     0    .  $$
\label{theorem_local}
\end{theorem}
\begin{proof}
The chains enumerated by the term
$\Psi(\Semi(\Lambda|_{F},n))$
fall into four cases:
(i)
the empty chain $\{\hz < \ho\}$,
(ii)
the chain $\{\hz < * < \ho\}$,
(iii)
chains containing non-trivial elements from $\Lambda$ but not containing $*$,
and
finally,
(iv)
chains containing $*$ and non-trivial elements from $\Lambda$.

The alternating sum of the weights of the chains of type (i)
is given by
\begin{equation}
\sum_{\hz \leq F \leq \ho}
          (-1)^{\rho(F,\ho)}
       \cdot
          (\av-\bv)^{n}
  =
0  ,
\label{chain_one}
\end{equation}
since the face poset of any regular subdivision
satisfies the classical Eulerian relation.
Similarly, the alternating sum of the weights of the chains
of type (ii) is given by
\begin{equation}
(-1)^{n+1} \cdot \left(1 + (-1)^{n-1}\right)
       \cdot
          (\av-\bv)^{n-1} \cdot \bv
     +
\sum_{\hz < F < \ho}
          (-1)^{\rho(F,\ho)}
       \cdot
          (\av-\bv)^{n-1} \cdot \bv
     =
0  .
\label{chain_two}
\end{equation}
Observe that the first term corresponds to $F = \hz$
(see Example~\ref{example_manifold})
and that there is no contribution from $F = \ho$.

Now consider chains of type (iii).
Note that there is no contribution from the term
$F = \hz$. From the remaining terms we have
the contribution
\begin{align}
  &   
\sum_{\hz < F \leq \ho}
      (-1)^{\rho(F,\ho)}
    \cdot
      \sum_{\onethingatopanother{\hz < x < \ho}{\sigma(x) \leq F}}
             \Psi([\hz,x]) \cdot \bv \cdot (\av-\bv)^{\rho(x,\ho)-1} 
\nonumber \\
  = &
\sum_{\hz < x < \ho}
\left(
\sum_{\sigma(x) \leq F \leq \ho}
      (-1)^{\rho(F,\ho)}
\right)
    \cdot
             \Psi([\hz,x]) \cdot \bv \cdot (\av-\bv)^{\rho(x,\ho)-1} 
    = 
0 ,
\label{chain_three}
\end{align}
since the inner sum is equal to zero.
Here we let $\sigma(x)$ denote the smallest dimensional face
in $\Gamma$ containing the face
$x \in \Lambda$.

Finally, consider the chains of type (iv).
Here the only contribution is from the terms
$\hz < F <\ho$:
\begin{align}
     &
   \sum_{\hz < F < \ho}
          (-1)^{\rho(F,\ho)}
       \cdot
          \sum_{\onethingatopanother{\hz < x < \ho}{\sigma(x) \leq F}}
             \Psi([\hz,x]) \cdot \bv \cdot \zetabar(x,*)
             \cdot (\av-\bv)^{\rho(x,*)-1}\cdot \bv 
\nonumber \\
   = &
\sum_{\hz < x < \ho}
\left(
\sum_{\sigma(x) \leq F < \ho}
          (-1)^{\rho(F,\ho)}
       \cdot
         \zetabar(x,*) 
\right)
       \cdot
\Psi([\hz,x]) \cdot \bv \cdot (\av-\bv)^{\rho(x,*)-1} \cdot \bv 
= 0 .
\label{chain_four}
\end{align}
Again we claim that the inner sum is equal to zero.
Observe that if we have the strict inequality $\sigma(x) < F$
then the face $x$ is on the boundary of $F$ and
$\link_{F}(x)$ is contractible,
that is, the reduced Euler characteristic is $0$
and the weighted zeta function is given by $\zetabar(x,*) = 1$.
On the other hand if the equality $\sigma(x) = F$ holds
then the face $x$ is on the interior $\sigma(x)$ and 
$\link_{\sigma(x)}(x)$ is a sphere of dimension
$\rho(\sigma(x)) - \rho(x) - 1$.
Its reduced Euler characteristic is $(-1)^{\rho(\sigma(x)) - \rho(x) - 1}$.
Hence we have
$\zetabar(x,*) 
   =
 1 + (-1)^{\rho(x,*)} \cdot (-1)^{\rho(\sigma(x)) - \rho(x) - 1}
   =
 1 + (-1)^{\rho(\sigma(x),\ho)}$
and the inner sum is given by
$$
(-1)^{\rho(\sigma(x),\ho)}
\cdot
\left(1 + (-1)^{\rho(\sigma(x),\ho)}\right)
+
\sum_{\sigma(x) < F < \ho}
          (-1)^{\rho(F,\ho)}
 =
\sum_{\sigma(x) \leq F \leq \ho}
          (-1)^{\rho(F,\ho)}
 =
0 .
$$
The Eulerian relation now follows by summing the
four identities
\eqref{chain_one},
\eqref{chain_two},
\eqref{chain_three} and \eqref{chain_four}.
\end{proof}

\section{Merging strata}
\label{section_merging_strata}

We now consider the operation of merging three strata and its
effect on the $\cd$-index.
This is the geometric
analogue of zipping elements in quasi-graded posets.
\begin{theorem}
Let $W$ be a Whitney stratification of a manifold $M$ such
that $M^{0}$ is one of the strata. 
Assume that $x$, $y$ and $z$ are three strata such that
$\dim(x) = \dim(y) = \dim(z)+1$ and 
when replacing the three strata $x$, $y$ and $z$
with a single stratum $w = x \cup y \cup z$
one obtains a Whitney stratification~$W^{\prime}$.
Then the $\cd$-index changes according to
$$  \Psi(W^{\prime})
   = 
    \Psi(W) - \Psi([\hz,z]) \cdot \dv \cdot \Psi([x,\ho]) . $$
\label{theorem_joining_two_strata}
\end{theorem}
\begin{proof}
Note that $W^{\prime} = W - \{x,y,z\} \cup \{w\}$.
We begin to show that
$x$, $y$ and $z$ form a zipper in the face poset of $W$.
Condition $(i)$ in Definition~\ref{definition_zipper}
follows from the dimension condition
in the statement of the theorem.

Since $W^{\prime}$ is a Whitney stratification,
we know that
the link of a point $p \in w$ is
independent of the choice of $p$ 
if the point $p$ belongs to $x$, $y$ or $z$.
Hence the strata $z$ is only covered by the strata~$x$ and $y$,
verifying condition $(ii)$.
By similar reasoning condition $(iii)$ follows.
Pick a point $p$ in $z$.
Locally the neighborhood of $p$ in $z$ is $\Rrr^{\dim(z)}$
and the neighborhood of $p$ in $w$ is $\Rrr^{\dim(z)+1}$.
Thus the neighborhood of $p$ in $x$ (or $y$)
is a half space. Hence $\link_x(z)$  is a point
and we conclude that $\zetabar(x,z) = 1$
(and $\zetabar(y,z) = 1$).

Next we verify that the zipped poset is
indeed the face poset of the stratification $W^{\prime}$.
All we need to verify is that their weighted zeta functions
agree, since they already
have the same poset structure and rank
function.
Observe that $\link_v(x)$ 
is the same as $\link_v(w)$
since $x$ is contained in~$w$.
Hence we have
$\zetabar_{W^{\prime}}(w,v) = \zetabar(x,v)$.

Next we must show
$\zetabar_{W^{\prime}}(u,w)
      = \zetabar(u,x) + \zetabar(u,y) - \zetabar(u,z)$
for all $u \in P$.
Here we have several cases to verify.
If the strata $u$ is comparable to
$x$ (or $y$) only, two terms are equal to zero
and the identity holds.
If the strata $u$ is less than $z$
(and hence $x$ and $y$) then
the identity
follows by the principle of inclusion-exclusion
for the Euler characteristic.
See Lemma~\ref{lemma_MV}.
Finally, if the strata~$u$ is less than $x$
and $y$, but not $z$, then $\zetabar(u,z) = 0$
and the identity follows by the additivity
of the Euler characteristic.

Hence the face poset of $W^{\prime}$ is the
result of zipping the face poset of $W$
and the identity follows from Proposition~\ref{proposition_zipping}.
\end{proof}

\section{Shelling components for non-pure simplicial complexes}
\label{section_shelling_components}

We now turn our attention to computing
the $\cd$-index of the $n$th semisuspension of a (non-pure)
shellable simplicial complex.
The first step is to define the $\cd$-index of the
simplicial shelling components.

For $i \leq k$ let $\Delta_{k,i}$ be the simplicial complex
consisting of $i+1$ facets of the $k$-dimensional simplex.
Define the quasi-graded poset $P_{n,k,i}$ for $0 \leq i \leq k \leq n$
to be the face poset of the semisuspension
$$ P_{n,k,i} = {\mathcal F}(\Semi(\Delta_{k,i},n)) . $$
Define the $\cd$-index of the simplicial shelling component,
$\shellcomponent(n,k,i)$, to be the difference
$$   \shellcomponent(n,k,i)
   =
     \Psi(P_{n,k,i}) - \Psi(P_{n,k,i-1})   \:\:\:
                \mbox{ for $1 \leq i \leq k \leq n$},
$$ 
and 
$\shellcomponent(n,k,0) = \Psi(P_{n,k,0})$
for $0 \leq k \leq n$.
See Table~\ref{table_shelling_component}
for the degree $2$ and $3$ cases.
Observe that
\begin{equation}
    \Psi(P_{n,k,i}) 
   =
     \sum_{j=0}^{i}
       \shellcomponent(n,k,j)   .  
\label{equation_shell_sum}
\end{equation}
The polynomials
$\shellcomponent(n,n,i)$
(the case $k=n$) were introduced by Stanley~\cite{Stanley_d}.

To obtain a recursion for the shelling components,
we need the following identity.
\begin{proposition}
For $0 \leq i \leq k \leq n$ the following identity holds:
\begin{equation}
\Psi(P_{n,k,i} \times B_{1})
+
\Psi(P_{n+1,k+1,0})
=
\Psi(P_{n+1,k+1,i+1})
+
\Psi(P_{n,k,i} * B_{2})   . 
\label{equation_four_posets}
\end{equation}
\label{proposition_four_posets}
\end{proposition}
\begin{proof}
View the only facet of $\Delta_{k+1,0}$ to be one of the facets
not among the facets of $\Delta_{k+1,i}$.
Then we have  the union
$\Delta_{k+1,i} \cup \Delta_{k+1,0} = \Delta_{k+1,i+1}$
and the intersection
$\Delta_{k+1,i} \cap \Delta_{k+1,0} = \Delta_{k,i}$.
Hence by the
inclusion-exclusion relation in
Theorem~\ref{theorem_inclusion-exclusion}, we have
$$  \Psi(\Semi(\Delta_{k+1,i},n+1))
   +
    \Psi(\Semi(\Delta_{k+1,0},n+1))
   =
    \Psi(\Semi(\Delta_{k+1,i+1},n+1))
   +
    \Psi(\Semi(\Delta_{k,i},n+1))  . $$
Or equivalently,
\begin{equation}
    \Psi(P_{n+1,k+1,i})
   +
    \Psi(P_{n+1,k+1,0})
   =
    \Psi(P_{n+1,k+1,i+1})
   +
    \Psi(P_{n+1,k,i})   .  
\label{equation_four_posets_1}
\end{equation}
Applying Proposition~\ref{proposition_Wednesday}
to the last term of~\eqref{equation_four_posets_1} gives
\begin{equation}
    \Psi(P_{n+1,k,i})
   =
    \Psi(P_{n,k,i}) \cdot \cv
   -
    \Psi([\hz,*]) \cdot \dv  ,  
\label{equation_four_posets_2}
\end{equation}
where the interval $[\hz,*]$ is in
the quasi-graded poset $P_{n,k,i}$.

Now observe that the quasi-graded poset
$P_{n,k,i} \times B_{1}$ is the face poset
of the cone of $\Semi(\Delta_{k,i},n)$.
This cone has two facets,
namely, $(\ho,\hz)$ and $(*,\ho)$,
and they share a subfacet $(*,\hz)$.
Replacing these three strata with
a new strata $*$, we obtain by
Theorem~\ref{theorem_joining_two_strata}
that
\begin{equation}
     \Psi(P_{n,k,i} \times B_{1})
   -
      \Psi([\hz,(*,\hz)]) \cdot \dv
   =
    \Psi(\Semi(\Delta_{k+1,i}, n+1))
   =
    \Psi(P_{n+1,k+1,i})  .    
\label{equation_four_posets_3}
\end{equation}
Observe that the interval $[\hz,(*,\hz)]$ in
$P_{n,k,i} \times B_{1}$ is identical to
the interval $[\hz,*]$ in $P_{n,k,i}$.

Finally, using 
$\Psi(P_{n,k,i} \times B_{1}) = \Psi(P_{n,k,i}) \cdot \cv$
by Lemma~\ref{lemma_Stanley_product},
adding equations~\eqref{equation_four_posets_1},
\eqref{equation_four_posets_2}
and~\eqref{equation_four_posets_3},
and canceling terms
yields
equation~\eqref{equation_four_posets}.
\end{proof}

The case $n=k$ in Proposition~\ref{proposition_four_posets}
appears in Theorem~8.1 in~\cite{Ehrenborg_Readdy}.
Refer to Section~\ref{section_poset_operations} for the definition
of the derivation $G$.

\begin{theorem}
The $\cd$-index of the simplicial shelling components
satisfy the recursion
$$ G(\shellcomponent(n,k,i)) = \shellcomponent(n+1,k+1,i+1) $$
with the boundary conditions
$$ \shellcomponent(n,k,0) = \Psi(B_{k}) \cdot \Psi(\Omega_{n-k+1}) , $$
for $k \geq 1$ and
$$ \shellcomponent(n,0,0)
    =
      \begin{cases}
        (\cv^{2}-2\dv)^{n/2}  
        & \text{  if $n$ is even,} \\
        (\cv^{2}-2\dv)^{(n-1)/2} \cdot \cv
        & \text{  if $n$ is odd.} \\
      \end{cases}
$$
\label{theorem_G}
\end{theorem}
\begin{proof}
Substituting the sum~\eqref{equation_shell_sum} into
Proposition~\ref{proposition_four_posets} yields
$$  \sum_{j=0}^{i} \Pyr(\shellcomponent(n,k,j))
+
\shellcomponent(n+1,k+1,0)
=
\sum_{j=0}^{i+1} \shellcomponent(n+1,k+1,j)
+
\sum_{j=0}^{i} \shellcomponent(n,k,j) \cdot \cv  .  $$
Using that $\Pyr(w) = w \cdot \cv + G(w)$ and canceling
terms, the identity simplifies to
\begin{equation}
\sum_{j=0}^{i} G(\shellcomponent(n,k,j))
=
\sum_{j=1}^{i+1} \shellcomponent(n+1,k+1,j) .  
\label{equation_G_sum}
\end{equation}
The recursion follows immediately from~\eqref{equation_G_sum}.
The boundary condition follows from
$P_{n,k,0} = B_{k} *  {\mathcal F}(\Omega_{n-k+1})$
when $k \geq 1$
and 
Example~\ref{example_manifold}
when $k=0$.
\end{proof}

We recall the notion of shellability of non-pure complexes
due to Bj\"orner and
Wachs~\cite{Bjorner_Wachs_non-pure_shellable_I}.
Let $\Delta$ be a non-pure simplicial complex, that is,
a simplicial complex with its maximal faces not necessarily all
of the same dimension.
We say $\Delta$ is {\em non-pure shellable} if
either
$\dim(\Delta) = 0$ (and hence any ordering of the facets
is a shelling order)
or
$\dim(\Delta) \geq 1$ 
and 
there is an ordering 
$F_{1}, \ldots, F_{m}$ of the facets,
called a {\em shelling order},
satisfying
$\left(\bigcup_{j=1}^{r-1} \overline{F_{j}}\right) \cap \overline{F_{r}}$
is pure of dimension
$\dim(F_{r}) - 1$
for $r = 2, \ldots, m$.
We say that a facet $F_{r}$ has shelling type
$(k,i)$ if $\dim(F_{r}) + 1 = k$ and
that the intersection 
$\left(\bigcup_{j=1}^{r-1} \overline{F_{j}}\right) \cap \overline{F_{r}}$
has $i$ facets (of dimension $k-2$).
Note that the first facet $F_{1}$ has type $(\dim(F_{1})+1,0)$.
Furthermore,
the only facet of shelling type $(k,0)$ is the first facet.
The $h$-triangle entry $h_{k,i}$ is the number of facets
of shelling type $(k,i)$.
It is a well-known result that
the $h$-triangle does not depend on the particular shelling order
and is equivalent to the 
$f$-triangle~\cite{Bjorner_Wachs_non-pure_shellable_I}.

\begin{theorem}
Let $\Delta$ be a non-pure shellable simplicial complex of
dimension at most $n$. Then the $\cd$-index of the semisuspension
of $\Delta$ is given by
$$
   \Psi(\Semi(\Delta,n))
 = 
   \sum_{k=0}^{n}
     \sum_{i=0}^{k}
          h_{k,i} \cdot \shellcomponent(n,k,i) .
$$
\label{theorem_h_triangle}
\end{theorem}
\begin{proof}
The proof is by induction on the number of facets in the complex $\Delta$.
If $\Delta$ has one facet, say of dimension $k$, then
${\mathscr F}(\Semi(\Delta,n)) = B_{k+1} * \Omega_{n-k} = P_{n,k,0}$.
Hence 
$\Psi(\Semi(\Delta,n)) = \shellcomponent(n,k,0)$
and
the only non-zero entry in the $h$-triangle is $h_{k,0} = 1$.

Now assume that the result is true for the complex $\Delta$
and we adjoin another facet $F$ of shelling type $(k,i)$.
Note that $i \geq 1$.
By the inclusion-exclusion
Theorem~\ref{theorem_inclusion-exclusion},
we have 
\begin{equation}
     \Psi(\Semi(\Delta \cup F,n))
    - 
     \Psi(\Semi(\Delta,n))
  =
     \Psi(\Semi(F,n))
    - 
     \Psi(\Semi(\Delta \cap F,n))   .
\label{equation_difference_of_semisuspensions}
\end{equation}
Again by inclusion-exclusion, the right-hand side
of~\eqref{equation_difference_of_semisuspensions}
is given by
$$   \Psi(\Semi(\Delta_{k,i},n))
    - 
     \Psi(\Semi(\Delta_{k,i-1},n))
  =
     \Psi(P_{n,k,i})
    - 
     \Psi(P_{n,k,i-1})
  =
     \shellcomponent(n,k,i)  ,  $$
completing the induction step.
\end{proof}

In the proof of
Theorem~\ref{theorem_h_triangle}
we did not use the fact that $F_{1}$ through $F_{m}$
are facets of the simplicial complex $\Delta$,
only that $F_{r}$ is a facet of 
$\overline{F_{1}} \cups \overline{F_{r-1}} \cup \overline{F_{r}}$.
Using this observation, we prove
the following Pascal type relation.
\begin{proposition}
For $0 \leq i \leq k < n$ we have the identity
$$   
    \shellcomponent(n,k+1,i)  
  =
    \shellcomponent(n,k,i)
  +
    \shellcomponent(n,k+1,i+1) .
$$
\end{proposition}
\begin{proof}
We begin to prove this identity for $i=0$.
Consider the simplex $\Delta$ having $k+1$ vertices.
Directly we have
$\Semi(\Delta,n) = \shellcomponent(n,k+1,0)$.
On the other hand we can build this complex in two steps,
first by adding a facet $F$ of cardinality $k$ and
next by adding the simplex $\Delta$.
The first step yields $\shellcomponent(n,k,0)$
and the second step $\shellcomponent(n,k+1,1)$,
proving the identity.
The case $i > 0$ follows by applying the derivation~$G$
$i$ times.
\end{proof}

\begin{table}
$$
\begin{array}{c | l c c}
k \backslash i 
& \multicolumn{1}{c}{0} & 1 & 2 \\ \hline
0 & \mcc - 2 \cdot \md
  & \hspace*{12 mm}   & \hspace*{12 mm} \\
1 & \mcc - \md         & \md \\
2 & \mcc               & \md      &  0
\end{array}
\hspace*{10 mm}  
\begin{array}{c | l l c c}
k \backslash i
& \multicolumn{1}{c}{0} & \multicolumn{1}{c}{1}
& \multicolumn{1}{c}{2} & \multicolumn{1}{c}{3}
\\ \hline
0 & \mccc - 2 \cdot \mdc 
  & \hspace*{15 mm}   & \hspace*{15 mm}   & \hspace*{15 mm} \\
1 & \mccc - \mdc - \mcd & \mdc - \mcd \\
2 & \mccc - \mcd        & \mdc
  & \mcd \\
3 & \mccc + \mdc        & \mdc + \mcd
  & \mcd                & 0 
\end{array}
$$
\caption{The $\cd$-polynomials $\shellcomponent(2,k,i)$
and $\shellcomponent(3,k,i)$.}
\label{table_shelling_component}
\end{table}

Recalling that the polynomials
$\shellcomponent(n,k,0)$
are given in Theorem~\ref{theorem_G},
we have the following expression for
the $\cd$-polynomial $\shellcomponent(n,k,i)$.
\begin{corollary}
We have the identity
$$   \shellcomponent(n,k,i)
   =
     \sum_{j=0}^{i}
         (-1)^{j}
             \cdot
         \binom{i}{j}
             \cdot
         \shellcomponent(n,k-j,0) . $$
\end{corollary}

\section{Concluding remarks}

Theorem~\ref{theorem_main_theorem} was developed using
an ambient smooth manifold $M$.
This result can be developed without the ambient manifold
using abstract stratifications. We decided against this
approach since it is unnecessarily complex and does
not give us any advantages.

As was mentioned in the introduction,
finding the
linear inequalities that hold among the entries of the $\cd$-index
of a Whitney stratified manifold
expands the program of determining linear inequalities
for flag vectors of polytopes.
Since the coefficients may be negative,
one must ask what should the new minimization inequalities be.
Observe that Kalai's convolution~\cite{Kalai} still holds. 
More precisely,
let $M$ and $N$ be two linear functionals
defined on the $\cd$-coefficients of any $m$-dimensional, respectively,
$n$-dimensional manifold. If both
$M$ and $N$ are non-negative  then their convolution
is non-negative on any $(m+n+1)$-dimensional manifold.

Other inequality questions are:
Can Ehrenborg's lifting technique~\cite{Ehrenborg_lifting}
be extended to stratified manifolds?
Is there an associated Stanley--Reisner ring for
the barycentric subdivision of a stratified space,
and if so,
what is the right version of the Cohen--Macaulay 
property~\cite{Stanley_green}?
Finally, what non-linear inequalities hold among the $\cd$-coefficients?

One interpretation of the coefficients of the
$\cd$-index is due to Karu~\cite{Karu}
who, for each $\cd$-monomial,
gave a sequence of operators on sheaves of vector spaces 
to show the non-negativity of the
coefficients of the $\cd$-index for Gorenstein* posets~\cite{Karu}.
Is there a signed analogue of Karu's construction to explain
the negative coefficients occurring in the $\cd$-index of
quasi-graded posets?

One of the original motivations for this paper was to develop
a local $\cd$-index in the spirit of Stanley's local
$h$-vector~\cite{Stanley_local}
which was used to understand the effect of subdivisions on the
$h$-vector.
The main result of Section~\ref{section_local}
implies that such a local $\cd$-index would always give the
zero polynomial.

Observe that the shelling components
$\shellcomponent(n,k,i)$ can have negative
coefficients.  Refer to Table~\ref{table_shelling_component}.
For which values of $n$, $k$ and $i$ do we know
that they are non-negative?
Is there a combinatorial interpretation
of the $\cd$-polynomials
$\shellcomponent(n,k,i)$?
In the case $n=k$ such interpretations
are known in terms of 
Andr\'e permutations and 
Simsun permutations~\cite{Ehrenborg_Readdy,Hetyei,Hetyei_Reiner}.

Given a Whitney stratified space,
its face poset with 
rank function given by dimension and weighted zeta function 
involving the Euler characteristic
(see~Definition~\ref{definition_main_definition}
and Theorem~\ref{theorem_main_theorem})
yield an Eulerian quasi-graded poset.
Conversely, given an Eulerian quasi-graded poset 
$(P,\rho,\zetabar)$ can one construct
an associated Whitney stratified space?
It is clear that for $x \coveredby y$ with
$\rho(x) + 1 = \rho(y)$ one must require 
$\zetabar(x,y)$ to be a positive integer since
$\link_y{x}$ is a $0$-dimensional
space consisting of a collection of one or more points.
What other conditions
on an Eulerian quasi-graded poset are necessary
so that it is the face poset of a Whitney stratified space?

As always when the $\ab$-index is defined
one also has the companion quasisymmetric function.
This quasisymmetric function can be defined by
the (almost) isomorphism $\gamma$
in~\cite[Section~3]{Ehrenborg_Readdy_Tchebyshev_transform}.
More directly, for a chain
$c = \{\hz = x_{0} < x_{1} < \cdots < x_{k} = \ho\}$,
define the composition
$$    \rho(c)
   =
        (\rho(x_{0},x_{1}),
         \rho(x_{1},x_{2}),
            \ldots,
         \rho(x_{k-1},x_{k}))   .  $$
Then the quasisymmetric function of
a quasi-graded poset $(P,\rho,\zetabar)$
is given by
$$
     F(P,\rho,\zetabar)
     =
     \sum_{c} \zetabar(c) \cdot M_{\rho(c)}  ,  
$$
where $M$ is the monomial quasisymmetric function.
It is straightforward to observe that $F$ can be
viewed as a Hopf algebra morphism
as follows.

\begin{theorem}
The following two identities hold for the
quasisymmetric function of a quasi-graded poset:
\begin{align*}
     \Delta(F(P,\rho,\zetabar))
  & = 
     \sum_{\hz \leq x \leq \ho}
               F([\hz,x],\rho,\zetabar)
          \tensor
               F([x,\ho],\rho,\zetabar)  ,    \\
F(P \times Q,\rho_{P \times Q},\zetabar_{P \times Q})
  & = 
F(P,\rho_{P},\zetabar_{P})
  \cdot
F(Q,\rho_{Q},\zetabar_{Q})  .
\end{align*}
\label{theorem_Hopf}
\end{theorem}
See~\cite{Aguiar_Bergeron_Sottile}
for results on generalized
Dehn--Sommerville relations in the setting of
combinatorial Hopf algebras.

\section*{Acknowledgements}

The authors thank the Institute for Advanced Study
where this research was done
and the two referees for their
careful comments.
The first author is partially supported by
National Science Foundation grant 0902063.
The second author is grateful to the Defense Advanced
Research Projects Agency for its support under 
DARPA grant number HR0011-09-1-0010.
This work was partially supported by a grant from the 
Simons Foundation (\#206001 to Margaret Readdy).
The first and third authors were
also partially supported by NSF grants
DMS-0835373 and CCF-0832797
and the second author by a Bell Companies Fellowship
for 2010.

\newcommand{\article}[6]{{\sc #1,} #2, {\it #3} {\bf #4} (#5), #6.}
\newcommand{\book}[4]{{\sc #1,} {\it #2,} #3, #4.}
\newcommand{\bookf}[5]{{\sc #1,} {\it #2,} #3, #4, #5.}
\newcommand{\preprint}[3]{{\sc #1,} #2, preprint #3.}
\newcommand{\web}[3]{{\sc #1,} #2, #3.}
\newcommand{\JCTA}{J.\ Combin.\ Theory Ser.\ A}
\newcommand{\AdvancesinMathematics}{Adv.\ Math.}
\newcommand{\JournalofAlgebraicCombinatorics}{J.\ Algebraic Combin.}

\newcommand{\appear}[3]{{\sc #1,} #2, to appear in {\it #3}}

\bigskip

\noindent
{\small \em
R.\ Ehrenborg,
M.\ Readdy,
Department of Mathematics,
University of Kentucky,
Lexington, KY 40506, \\
\{jrge,readdy\}@ms.uky.edu
}

\noindent
{\small \em
M.\ Goresky,
School of Mathematics,
Institute for Advanced Study,
Princeton, NJ 08540, \\
goresky@math.ias.edu
}


{\small

\begin{thebibliography}{99}

\bibitem{Aguiar_Bergeron_Sottile}
\article{M.\ Aguiar, N.\ Bergeron and F.\ Sottile}
        {Combinatorial Hopf algebras and generalized
         Dehn--Sommerville relations}
        {Compos.\ Math}
        {142}{2006}{1--30}

\bibitem{Bayer_Billera}
\article{M.\ Bayer and L.\ Billera}
        {Generalized Dehn--Sommerville relations for polytopes,
         spheres and Eulerian partially ordered sets}
        {Invent.\ Math.}
        {79}{1985}{143--157}

\bibitem{Bayer_Klapper}
\article{M.\ Bayer and A.\ Klapper}
        {A new index for polytopes}
        {Discrete Comput.\ Geom.}
        {6}{1991}{33--47}

\bibitem{Billera_Brenti}
\article{L.\ J.\ Billera and F.\ Brenti}
        {Quasisymmetric functions and Kazhdan--Lusztig polynomials}
        {Israel J.\ Math.}
        {184}{2011}{317--348} 

\bibitem{Billera_Ehrenborg}
\article{L.\ J.\ Billera and R.\ Ehrenborg}
        {Monotonicity properties of the ${\bf cd}$-index for polytopes}
        {Math.\ Z.}
        {233}{2000}{421--441}



\bibitem{Billera_Ehrenborg_Readdy_om}
\article{L.\ J.\ Billera, R.\ Ehrenborg, and M.\ Readdy}
        {The $\ctd$-index of oriented matroids}
        {\JCTA}
        {80}{1997}{79--105} 



\bibitem{Billera_Lee}
\article{L.\ J.\ Billera and C.\ W.\ Lee}
        {A proof of the sufficiency of McMullen's conditions
         for $f$-vectors of simplicial polytopes}
        {\JCTA}
        {31}{1981}{237--255} 



\bibitem{Bjorner_handbook}
           {\sc A.\ Bj\"orner,}
           Topological methods.
           Handbook of combinatorics,
           Vol.\ 1, 2, 1819--1872
           Elsevier, Amsterdam, 1995.


\bibitem{Bjorner_Korte_Lovasz}
\article{A.\ Bj\"orner, B.\ Korte and L.\ Lov\'asz}
        {Homotopy properties of greedoids}
        {Adv.\ in Appl.\ Math}
        {6}{1985}{447--494}


\bibitem{Bjorner_Wachs_non-pure_shellable_I}
\article{A.\ Bj\"orner and M.\ Wachs}
        {Shellable nonpure complexes and posets. I}
        {Trans.\ Amer.\ Math.\ Soc.}
        {348}{1996}{1299--1327}


\bibitem{Borsuk}
\article{K.\ Borsuk}
        {On the imbedding of systems of compacta in simplicial complexes}
        {Fund.\ Math.}
        {35}{1948}{217--234}


\bibitem{Dehn}
\article{M.\ Dehn}
        {Die Eulersche Formel in Zusammenhang mit dem Inhalt
         in der Nicht-Euklidischen Geometrie}
        {Math.\ Ann.}
        {61}{1906}{561--586}


\bibitem{Denkowska}
\article{Z.\ Denkowska, K.\ Wachta and J.\ Stasica}
        {Stratifications des ensembles sous-analytiques avec les
         propri\'et\'es (A) et (B) de Whitney}
        {Univ.\ Iagel.\ Acta Math.}
        {25}{1985}{183--188}


\bibitem{du_Plessis_Wall}
\book{A.\ du Plessis and T.\ Wall}
     {The Geometry of Topological Stability}
     {London Mathematical Society Monographs. New Series, 9.
      Oxford Science Publications. The Clarendon Press,
      Oxford University Press, New York}
     {1995}


\bibitem{Ehrenborg_lifting}
\article{R.\  Ehrenborg}
        {Lifting inequalities for polytopes}
        {Adv.\ Math.}
        {193}{2005}{205--222}
 



\bibitem{Ehrenborg_Johnston_Rajagopalan_Readdy}
\article{R.\ Ehrenborg, D.\ Johnston, R.\ Rajagopalan
         and M.\ Readdy}
        {Cutting polytopes and flag $f$-vectors}
        {Discrete Comput.\ Geom.}
        {23}{2000}{261--271}

\bibitem{Ehrenborg_Karu}
\article{R.\ Ehrenborg and K.\ Karu}
        {Decomposition theorem for the $\cd$-index of Gorenstein posets}
        {J.\ Algebraic Combin.}
        {26}{2007}{225--251}




\bibitem{Ehrenborg_Readdy_Sheffer}
\article{R.\ Ehrenborg and M.\ Readdy}
        {Sheffer posets and $r$-signed permutations}
        {Ann.\ Sci.\ Math.\ Qu\'ebec}
        {19}{1995}{173--196}



\bibitem{Ehrenborg_Readdy}
\article{R.\ Ehrenborg and M.\ Readdy}
        {Coproducts and the $\cd$-index}
        {J.\ Algebraic Combin.}
        {8}{1998}{273--299}


\bibitem{Ehrenborg_Readdy_Eulerian_binomial}
\article{R.\ Ehrenborg and M.\ Readdy}
        {Classification of the factorial functions of
         Eulerian binomial and Sheffer posets}
        {\JCTA}
        {114}{2007}{339--359}

\bibitem{Ehrenborg_Readdy_Tchebyshev_transform}
\article{R.\ Ehrenborg and M.\ Readdy}
        {The Tchebyshev transforms of the first and second kind}
        {Ann.\ Comb.}
        {14}{2010}{211--244}

\bibitem{Ehrenborg_Readdy_Bruhat}
\preprint{R.\ Ehrenborg and M.\ Readdy}
         {The $\cd$-index of Bruhat and balanced graphs}
         {2013}
         arXiv:1304.1169

\bibitem{Ehrenborg_Readdy_Slone}
\article{R.\ Ehrenborg, M.\ Readdy and M.\ Slone}
        {Affine and toric hyperplane arrangements}
        {Discrete Comput.\ Geom.}
        {41}{2009}{481--512}


\bibitem{Hetyei}
\article{G.\ Hetyei}
        {On the $cd$-variation polynomials of
         Andr\'e and simsun permutations}
        {Discrete Comput.\ Geom.}
        {16}{1996}{259--276}

\bibitem{Hetyei_Reiner}
\article{G.\ Hetyei and E.\ Reiner}
        {Permutation trees and variation statistics}
        {European J.\ Combin.}
        {19}{1998}{847--866}



\bibitem{Gibson_Wirthmuller_du_Plessis_Looijenga}
\book{C.\ G.\ Gibson, K.\ Wirthm\"uller,
      A.\ A.\ du Plessis and E.\ J.\ N.\ Looijenga}
     {Topological Stability of Smooth Mappings}
     {Lecture Notes in Mathematics, Vol. 552.
       Springer--Verlag, Berlin--New York}
     {1976}



\bibitem{Goresky_MacPherson}
\book{M.\ Goresky and R.\ MacPherson}
     {Stratified Morse Theory, 
     Ergebnisse der Mathematik und ihrer Grenzgebiete (3) 
     [Results in Mathematics and Related Areas (3)], 14}
     {Springer--Verlag, Berlin}
     {1988}


\bibitem{Hardt}
\article{R.\ M.\ Hardt}
        {Stratifications of real analytic mappings and images}
        {Invent.\ Math.}
        {28}{1975}{193--208}

\bibitem{Hironaka1}
        {\sc H.\ Hironaka,}
        Subanalytic sets, in {Number Theory, Algebraic Geometry 
        and Commutative Algebra}, volume in honor of Y.\ Akizuki, 
        Kinokuniya Tokyo 1973, pp.\ 453--493.

\bibitem{Hironaka2}
        {\sc H.\ Hironaka,}
        Stratification and flatness, in {Real and Complex Singularities},
        Nordic summer school (Oslo, 1976), Sijthoff-Noordhoff, Groningen,
        1977, pp.\ 199--265. 

\bibitem{Kalai}
\article{G.\ Kalai}
        {A new basis of polytopes}
        {\JCTA}
        {49}{1988}{191--209}

\bibitem{Karu}
\article{K.\ Karu}
        {The $cd$-index of fans and posets}
        {Compos.\ Math.}
        {142}{2006}{701--718}


\bibitem{Klain_Rota}
\book{D.\ Klain and G.-C.\ Rota}
     {Introduction to Geometric Probability.  
      Lezioni Lincee. [Lincei Lectures]}
     {Cambridge University Press}
     {Cambridge, 1997}


\bibitem{Lojasiewicz}
\article{S.\ Lojasiewicz}
           {Triangulation of semi-analytic sets}
           {Ann.\ Scuola Norm.\ Sup.\ Pisa (3)}
           {18}{1964}{449--474}

\bibitem{Mather}
        {\sc J.\ Mather,}
         Notes on topological stability,
         Harvard University,
         1970,
         in
          Bull.\ Amer.\ Math.\ Soc.\ (N.S.)
          {\bf 49} (2012), 475--506. 

\bibitem{Mather2}
        {\sc J.\ Mather,}
        Stratifications and mappings. Dynamical systems
        (Proc.\ Sympos., Univ.\ Bahia, Salvador, 1971),
        pp.\ 195--232. Academic Press, New York, 1973.


\bibitem{McMullen}
\article{P.\ McMullen}
        {The maximum numbers of faces of a convex polytope}
        {Mathematika}
        {17}{1970}{179--184}



\bibitem{Mulmuley}
\article{K.\ Mulmuley}
        {A generalization of Dehn--Sommerville relations 
         to simple stratified spaces}
        {Discrete Comput.\ Geom.}       
        {9}{1993}{47--55}




\bibitem{Reading}
\article{N.\ Reading}
        {The $cd$-index of Bruhat intervals}
        {Electron.\ J.\ Combin.}
        {11}{2004}{R74, 25pp}



\bibitem{Rota}
        {\sc G.-C.\ Rota},
        {On the combinatorics of the Euler characteristic},
        in
        {\em Studies in Pure Mathematics (Presented to Richard Rado)},
         pp.\ 221--233, Academic Press, London, 1971.

\bibitem{Rota_Foundations_I}
\article{G.-C.\ Rota}
        {On the foundations of combinatorial theory. I. 
        Theory of M\"obius functions}
        {Z.\ Wahrscheinlichkeitstheorie und Verw.\ Gebiete}
        {2}{1964}{340--368}

\bibitem{Schanuel}
        {\sc S.\ Schanuel},
        {Negative sets have Euler characteristic and dimension}
        in
        {\em Category theory (Como, 1990),
        Lecture Notes in Math., 1488},
        pp.\ 379--385,
        Springer, Berlin, 1991. 


\bibitem{Sommerville}
\article{D.\ M.\ Y.\ Sommerville}
        {The relations connecting the angle-sums and volume
         of a polytope of a polytope in space of $n$ dimensions}
        {Proc.\ Roy.\ Soc.\ London, Ser A}
        {115}{1927}{103--119}


\bibitem{Stanley_g}
\article{R.\ P.\ Stanley} 
        {The number of faces of a simplicial convex polytope}
        {\AdvancesinMathematics}
        {35}{1980}{236--238}
     


\bibitem{Stanley_local}
\article{R.\ P.\ Stanley}
        {Subdivisions and local $h$-vectors}
        {J.\ Amer.\ Math.\ Soc.}
        {5}{1992}{805--851}



\bibitem{Stanley_d}
\article{R.\ P.\ Stanley}
        {Flag $f$-vectors and the $cd$-index}
        {Math.\ Z.}
        {216}{1994}{483--499}




\bibitem{Stanley_green}
\bookf{R.\ P.\ Stanley}
      {Combinatorics and Commutative Algebra, Second edition}
      {Progress in Mathematics, 41}
      {Birkh\"auser Boston, Inc.}
      {Boston, MA, 1996}

\bibitem{Stanley_EC_1}
\book{R.\ P.\ Stanley}
         {Enumerative Combinatorics, Vol 1, second edition}
         {Cambridge University Press, Cambridge}
         {2012}



\bibitem{Steinitz}
\article{E.\ Steinitz}
        {\"Uber die Eulerischen Polyderrelationen}
        {Arch.\ Math.\ Phys.}
        {11}{1906}{86--88}






\bibitem{Thom}
\article{R.\ Thom}
        {Ensembles et morphismes stratifi\'es}
        {Bull.\ Amer.\ Math.\ Soc.}
        {75}{1969}{240--284}

\bibitem{Verdier}
\article{J.\ L.\ Verdier}
        {Stratifications de Whitney et th\'eor\`eme de Bertini-Sard}
        {Inv.\ Math.}
        {36}{1976}{295--312}



\bibitem{Whitney}
    {\sc H.\ Whitney,}
    Local properties of analytic varieties.
    1965
    {\it Differential and Combinatorial Topology
    (A Symposium in Honor of Marston Morse)}
    pp. 205--244
    Princeton University Press,
    Princeton, NJ.



\end{thebibliography}

}
\end{document}